\documentclass[11pt,a4paper]{article}

\usepackage{a4}
\usepackage[utf8]{inputenc}
\usepackage[english]{babel}

\usepackage{amsmath, amsthm}
\usepackage{amsfonts,amssymb}
\usepackage{mathrsfs,upgreek}
\usepackage{dsfont,xfrac}
\usepackage{thmtools,yhmath}
\usepackage{authblk}
\usepackage{geometry}
\usepackage[dvipsnames]{xcolor}

\usepackage{hyperref}

\usepackage{chngcntr}

\newcommand{\N}{\mathbb{N}}

\newcommand{\Q}{\mathbb{Q}}
\newcommand{\R}{\mathbb{R}}
\newcommand{\C}{\mathbb{C}}
\newcommand{\E}{\mathbb{E}}

\newcommand{\rme}{{\rm e}}
\newcommand{\rmi}{{\rm i}}
\newcommand{\rmC}{{\rm C}}
\newcommand{\rmD}{{\rm D}}

\newcommand{\1}{\mathds{1}}

\newcommand{\Fc}{\mathcal{F}}

\newcommand{\Gc}{\mathcal{G}}

\newcommand{\Mc}{\mathcal{M}}
\newcommand{\Kc}{\mathcal{K}}
\newcommand{\Uc}{\mathcal{U}}
\newcommand{\Mcons}{\Mc_{\rm c}}

\newcommand{\vect}{{\rm vect}}

\renewcommand{\mid}{~\middle | ~}
\newcommand{\eps}{\varepsilon}
\newcommand{\mycdot}{\cdot}
\newcommand{\cv}[2][]{\underset{#2}{\overset{#1}{\longrightarrow}}}
\newcommand\quotient[2]{
        \mathchoice
            {
                \text{\raise1ex\hbox{$#1$}\Big/\lower1ex\hbox{$#2$}}%
            }
            {
                #1\,/\,#2
            }
            {
                #1\,/\,#2
            }
            {
                #1\,/\,#2
            }
}
\newcommand{\cadlag}{cadlag }
\renewcommand{\P}{\mathbb{P}}
\newcommand{\Pcal}{\mathcal{P}}
\newcommand{\D}{\mathbb{D}}
\newcommand{\Tbb}{\mathbb{T}}
\newcommand{\loc}{\text{loc}}
\newcommand{\Dloc}{\D_\loc}

\renewcommand{\d}{{\rm d}}
\newcommand{\Pbf}{\mathbf{P}}
\newcommand{\Ebf}{\mathbf{E}}
\newcommand{\Qbf}{\mathbf{Q}}
\newcommand{\Rbf}{\mathbf{R}}
\renewcommand{\vect}{{\rm \text{span}}}

\newcommand{\law}{\mathscr{L}}

\usepackage[many]{tcolorbox}
\newcounter{myhypo}

\newtcolorbox{hypo}[1]{
  breakable,
  enhanced,
  top=0pt,
  bottom=0pt,
  nobeforeafter,
  colback=white,
  boxrule=0pt,
  arc=0pt,
  right=40pt,
  left=20pt,
  outer arc=0pt,
  overlay={
    \node[inner sep=0pt,anchor=east] 
    at (frame.east) 
    {(#1)};
  },
}

\makeatletter
\newcommand{\pushright}[1]{\ifmeasuring@#1\else\omit\hfill$\displaystyle#1$\fi\ignorespaces}
\newcommand{\pushleft}[1]{\ifmeasuring@#1\else\omit$\displaystyle#1$\hfill\fi\ignorespaces}
\makeatother

\theoremstyle{plain}
\newtheorem{theorem}{Theorem}[section]
\newtheorem{lemma}[theorem]{Lemma}
\newtheorem{proposition}[theorem]{Proposition}

\theoremstyle{definition}
\newtheorem{definition}[theorem]{Definition}

\newtheorem{remark-base}[theorem]{Remark}

\theoremstyle{remark}

\counterwithin*{equation}{section}

\newenvironment{remark}{\pushQED{\qed}\begin{remark-base}}{\popQED\end{remark-base}}

\begin{document}
\title{Locally Feller processes and  martingale local problems}
\author{Mihai Gradinaru}
\author{Tristan Haugomat}
\affil{\small Institut de Recherche Math{\'e}matique de Rennes, Universit{\'e} de Rennes 1,\\Campus de Beaulieu, 35042 Rennes Cedex, France\\
\texttt{\small \{Mihai.Gradinaru,Tristan.Haugomat\}@univ-rennes1.fr}}
\date{
}
\maketitle

{\small\noindent {\bf Abstract:}~This paper is devoted to the study of a 
certain type of martingale problems associated to general operators corresponding 
to processes which have finite lifetime. We analyse several properties and 
in particular the weak convergence of sequences of solutions for an appropriate Skorokhod topology setting. 
We point out the Feller-type features of the associated solutions to this type of 
martingale problem. 
Then localisation theorems for well-posed martingale problems or for corresponding 
generators are proved.}
{\small\noindent{\bf MSC2010 Subject Classification:}~Primary~60J25; Secondary~60G44, 60J35, 60B10, 60J75, 47D07}


\section{Introduction}

The theory of Lévy-type processes stays an active domain of research during the last two decades.
Heuristically, a Lévy-type process $X$ with symbol $q:\R^d\times\R^d\to\C$ 
is a Markov process which behaves locally like a Lévy process with characteristic 
exponent $q(a,\cdot)$, in a neighbourhood of each point $a\in\R^d$.
One associates to a Lévy-type process the pseudo-differential operator $L$ given 
by, for $f\in\rmC_c^{\infty}(\R^d)$,
\begin{equation*}
Lf(a):=-\int_{\R^d}{\rme}^{\rmi a\mycdot\alpha}
q(a,\alpha){\widehat f}(\alpha)\d\alpha,\quad
\mbox{where}\quad
{\widehat f}(\alpha):=(2\pi)^{-d}\int_{\R^d}\rme ^{-\rmi a\mycdot\alpha}f(a)\d a.
\end{equation*}

Does a sequence $X^{(n)}$ of Lévy-type processes, having symbols $q_n$,
converges toward some process, when the sequence of symbols $q_n$ converges 
to a symbol $q$? What about a sequence $X^{(n)}$, corresponding to
operators $L_n$, when the sequence of operators  converges to an operator $L$?
What could be the appropriate setting when one wants to approximate a Lévy-type processes 
by a family of discrete Markov chains? 
This is the kind of question which naturally appears when we study the Lévy-type processes.

It was a very useful observation that a unified manner to tackle a lot of questions 
about large classes of processes is the martingale problem approach (see, for instance, 
Stroock \cite{St75} for Lévy-type processes, Stroock and Varadhan \cite{SV06} for diffusion processes,  
Kurtz \cite{Ku11} for Lévy-driven stochastic differential equations...). 
Often,  convergence results are obtained under technical restrictions: for instance, when
the closure of $L$ is the generator of a Feller process (see Kallenberg \cite{Ka02}  Thm. 19.25, p. 385, Thm. 19.28, p. 387 or 
B\"ottcher, Schilling and Wang \cite{BSW13}, Theorem 7.6 p. 172). In a number of situations 
the cited condition is not satisfied. 

In the present paper we describe a general method which should be the main tool to 
tackle these difficulties and, even, should  relax some of these technical restrictions. 
We analyse sequences of martingale problems associated to large class 
of  operators acting on continuous functions and we look to Feller-type features of the associated of solutions. 

In order to be more precise, let us point out that the local Skorokhod 
topology on a locally compact Hausdorff space $S$
constitutes a good setting when one needs to consider explosions in finite time
 (see \cite{GH017}). Heuristically, we modify the global Skorokhod topology, on the space 
 of \cadlag paths, by localising with respect to space variable, in order to include the eventual explosions. 
 The definition of a martingale local problem 
follows in a natural way: we need to stop the martingale when it exits from compact sets.  
Similarly, a stochastic process is locally Feller if, for any compact set of $S$,  it coincides with a Feller process before it exits from the compact set. 
Let us note that an useful tool allowing to make the connection between local and global 
objects (Skorokhod topology, martingale, infinitesimal generator or Feller processes) is the time change transformation. Likewise, one has stability of all these local notions under the time change. 

We study the existence and the uniqueness of solutions for martingale local problems 
and we illustrate their locally Feller-type features (see Theorem \ref{thmDefLFF}). Then we deduce a 
description of the generator of a locally Feller family of probabilities by using a martingale  (see Theorem \ref{thmGenDesc} below). Furthermore we characterise  the convergence of a sequence of locally Feller processes in terms of convergence of 
operators, provided that the sequence of martingale local problems are well-posed (see Theorem \ref{thmCvgLocFel} below) and without supposing that the closure of the limit 
operator is an infinitesimal generator. 
We also consider the localisation question (as described in Ethier and Kurtz \cite{EK86}, \S 4.6, pp. 216-221) and we give 
answers in terms of martingale local problem or in terms of generator (Theorems \ref{thmLocMP} and \ref{thmLocGen}). We stress that a Feller process is locally Feller, 
hence our results, in particular the convergence theorems apply to Feller processes. In Theorem \ref{SecondCharFeller} we give a characterisation of Feller property in terms of locally Feller property plus an additional condition.

Our results should be useful in several situations, for instance, to analyse the convergence 
of a Markov chain toward a Lévy-type process under general conditions (improving the results, for instance, Thm 11.2.3 from  Stroock and Varadhan \cite{SV06} p. 272, Thm. 19.28 from Kallenberg \cite{Ka02}, p. 387 or from B\"otcher and Schnurr \cite{BS11}). We develop some of these applications (as the Euler scheme of approximation for Lévy-type process
or the convergence of Sinai's random walk toward the Brox diffusion) in a separate work \cite{GH217}. 
The method which we develop should apply for other situations. 
In a work in progress, we try to apply a similar method for 
some singular stochastic differential equations driven by  $\alpha$-stable processes other than 
Brownian motion.

The present paper is organised as follows:  in the next section we recall some notations and results obtained in our 
previous paper \cite{GH017} on the local Skorokhod topology on spaces of \cadlag functions, tightness and time change 
transformation. Section 3 is devoted to the study of the martingale local problem : properties, tightness and convergence, 
but also the existence of solutions. The most important results are 
presented in Section 4. In \S 4.1 and \S 4.2 we give the definitions and point out characterisations of a 
locally Feller family and its connection with a Feller family, essentially in terms of  martingale local problems. 
We also provide two corrections of a  result by van Casteren 
\cite{vC92}. In \S 4.3 we give a generator description of a locally Feller family and 
we characterise the convergence of a sequence of locally Feller families. \S 4.4 contains the 
localisation procedure for martingale problems and generators.  We collect in the Appendix the most part of technical proofs. 

\section{Preliminary notations and results}

We recall here some notations and results concerning the local Skorokhod topology, the tightness criterion and a time change transformation  
which will be useful to state and prove our main results. Complete statements and proofs are described in a entirely dedicated paper \cite{GH017}. 

Let $S$ be a locally compact Hausdorff space with countable base. The space $S$ could be endowed with a metric and so it is a Polish space. 
Take $\Delta\not\in S$, and we will denote by $S^\Delta\supset S$ the one-point compactification of $S$, if $S$ is not compact, 
or the topological sum $S\sqcup\{\Delta\}$, if $S$ is compact  (so $\Delta$ is an isolated point).  Denote $\rmC(S):=\rmC(S,\R)$, resp. $\rmC(S^\Delta):=\rmC(S^\Delta,\R)$,  the set of real continuous functions on $S$, resp. on $S^\Delta$.
If $\rmC_0(S)$ denotes the set of functions $f\in\rmC(S)$ vanishing in $\Delta$, we will identify
\[
\rmC_0(S)=\left\{f\in\rmC(S^\Delta)\mid f(\Delta)=0\right\}.
\]
We endow the set $\rmC(S)$ with the topology of uniform convergence on compact sets and  $\rmC_0(S)$ with the topology of uniform convergence.

The fact that a subset $A$ is compactly embedded in an open subset $U\subset S$ will be denoted $A\Subset U$.
If  $x\in (S^\Delta)^{\R_+}$ we denote 
\[
\xi(x):=\inf\{t\geq 0\,|\,\{x_s\}_{s\leq t}\not\Subset S\}.
\]

Firstly, we introduce the set of \cadlag paths with values in $S^\Delta$,
\begin{equation*}
\D(S^\Delta):=\left\{x\in(S^\Delta)^{\R_+}\mid
\begin{array}{l}\forall t\geq 0,\;x_t=\lim_{s\downarrow t}x_s,\;\mbox{ and }\; \\
\forall t>0,\;x_{t-}:=\lim_{s\uparrow t}x_s~\text{ exists in  }S^\Delta\end{array}\right\}\,.
\end{equation*}
endowed with the global Skorokhod topology (see, for instance, Chap. 3 in \cite{EK86}, pp. 116-147) which is Polish. 
A sequence $(x^k)_k$ in $\D(S^\Delta)$ converges
to $x$ for the latter topology if and only if there exists a sequence $(\lambda^k)_k$ of increasing homeomorphisms on $\R_+$ such that
\[
\forall t\geq 0,\quad\lim_{k\to\infty}\sup_{s\leq t}d(x_s,x^k_{\lambda^k_s})=0
\quad
\mbox{and}\quad
\lim_{k\to\infty}\sup_{s\leq t}|\lambda^k_s -s|=0.
\]
The  global Skorokhod topology does not depend on the arbitrary  metric $d$ on $S^\Delta$, but only on the topology on $S$. 

Secondly, we proceed with the definition of a set of exploding \cadlag paths  
\begin{equation*}
\Dloc(S):=\left\{ x\in(S^\Delta)^{\R_+}\mid \begin{array}{l}
\forall t\geq\xi(x),~ x_t=\Delta,\\
\forall t\geq 0,~ x_t=\lim_{s\downarrow t}x_s,\\
\forall t>0\text{ s.t. }\{x_s\}_{s<t}\Subset S,~ x_{t-}:=\lim_{s\uparrow t}x_s\text{ exists}
\end{array}\right\},
\end{equation*}
endowed with the local Skorokhod topology (see Theorem 2.6 in\cite{GH017}) 
which is Polish. Similarly, a sequence $(x^k)_{k\in\N}$ in $\Dloc(S)$
converges to $x$ for the local Skorokhod topology if and only if there exists a sequence  $(\lambda^k)_k$ 
of increasing homeomorphisms on $\R_+$ satisfying 
\[
\forall t\geq 0\mbox{ s.t. }\{x_s\}_{s<t}\Subset S,\quad \lim_{k\to\infty}\sup_{s\leq t}d(x_s,x^k_{\lambda^k_s})=0
\quad\mbox{and}\quad
\lim_{k\to\infty}\sup_{s\leq t}|\lambda^k_s -s|=0.
\]
Once again, the local Skorokhod topology does not depend on the arbitrary  metric $d$ on $S^\Delta$, but only on the topology on $S$.

We will always denote by $X$ the canonical process on $\D(S^\Delta)$ or on $\Dloc(S)$, 
without danger of confusion.  We endow each of $\D(S^\Delta)$ and $\Dloc(S)$ with the 
Borel 
$\sigma$-algebra $\Fc:=\sigma(X_s,~0\leq s <\infty)$ and a filtration $\Fc_t:=\sigma(X_s,~0\leq s \leq t)$. 
We will always omit the argument $X$ for the explosion time $\xi(X)$ of the  canonical process. It is clear 
that $\xi$ is a stopping time. Furthermore, if $U\subset S$ is an open subset, 
\begin{equation}\label{eqtauU}
\tau^U:=\inf\left\{t\geq 0\mid X_{t-}\not\in U\text{ or }X_t\not\in U\right\}\wedge\xi
\end{equation}
is a stopping time.

In \cite{GH017} we state and prove the following version of the Aldous criterion of tightness:  
let $(\Pbf_n)_n$ be a sequence of probability measures on $\Dloc(S)$. If for all $t\geq 0$, $\eps>0$, and open subset $U\Subset S$, 
we have:
\begin{equation}\label{hypropAldous}
\limsup_{n\to\infty}\sup_{\substack{\tau_1\leq\tau_2\\ \tau_2\leq (\tau_1+\delta)\wedge t\wedge\tau^U}}\Pbf_n\big(d(X_{\tau_1},X_{\tau_2})\geq\eps\big)\cv{\delta\to 0}0,
\end{equation}
then $\{\Pcal_n\}_n$ is tight for the local Skorokhod topology. Here $d$ is an arbitrary metric on $S^\Delta$ and the supremum is taken on all stopping times $\tau_i$.

There are several ways to localise processes, for instance one can stop when they leave a large compact set. Nevertheless this method does not preserve the  convergence and we need to adapt this procedure in order to 
recover continuity. Let us describe our time change transformation. Since \eqref{eqtauU}, we can write
\[
\tau^{^{\{g\neq 0\}}}(x):=\inf\left\{t\geq 0\mid g(x_{t-})\wedge g(x_t)=0\right\}\wedge\xi(x).
\]

Let $g\in\rmC(S,\R_+)$ be. For any $x\in\Dloc(S)$ and $t\in\R_+$ we denote
\begin{align}\label{taug}
\tau_t^g(x):=\inf\left\{s\geq 0\mid s\geq\tau^{\{g\neq 0\}}\text{ or } \int_0^s\frac{\d u}{g(x_u)}\geq t\right\}.
\end{align}
We define a time change transformation, which is $\Fc$-measurable,
\[\begin{array}{cccc}
g\mycdot X:&\Dloc(S)&\rightarrow & \Dloc(S)\\
&x&\mapsto & g\mycdot x,
\end{array}\]
as follows:  for $t\in\R_+$
\begin{align}\label{eqdefTC}
(g\mycdot X)_t:=\left\{\begin{array}{ll}
X_{\tau^{^{\{g\neq 0\}}}-} & \text{if }\tau_t^g=\tau^{^{\{g\neq 0\}}},~X_{\tau^{^{\{g\neq 0\}}}-}\text{ exists and belongs to }\{g=0\},\\
X_{\tau_t^g} & \text{otherwise}.
\end{array}\right.
\end{align}
For any $\Pbf\in\Pcal(\Dloc(S))$, we also define $g\mycdot\Pbf$ the pushforward of $\Pbf$ by $x\mapsto g\mycdot x$.
Let us stress that, $\tau_t^g$ is a stopping time (see Corollary 2.3 in \cite{GH017}).
The time of explosion of $g\mycdot X$ is given by 
\begin{equation*}
\xi(g\mycdot X) = \left\{\begin{array}{ll}
\infty & \text{if }\tau^{^{\{g\neq 0\}}}<\xi\text{ or }X_{\xi-}\text{ exists and belongs to }\{g=0\},\\
\int_0^\xi\frac{\d u}{g(x_u)} & \text{otherwise}.
\end{array}\right.
\end{equation*}
It is not difficult to see, using the definition of the time change \eqref{eqdefTC}, that
\begin{equation}\label{eqProp2TC}
\forall g_1,g_2\in\rmC(S,\R_+),~\forall x\in \Dloc(S),\quad g_1\mycdot(g_2\mycdot x)= (g_1g_2)\mycdot x.
\end{equation}

In \cite{GH017}) Proposition 3.8, a connection  between $\Dloc(S)$ and $\D(S^\Delta)$
was given. We recall here this result because it will employed several times.
 \begin{proposition}[Connection  between $\Dloc(S)$ and $\D(S^\Delta)$]\label{propEqLocGlo}
Let $\widetilde{S}$ be an arbitrary locally compact Hausdorff space with countable base and consider
\[\begin{array}{cccc}
\Pbf:&\widetilde{S}&\to&\Pcal(\Dloc(S))\\
&a&\mapsto&\Pbf_a
\end{array}\]
a weakly continuous mapping for the local Skorokhod topology. Then for any open subset $U$ of $S$, there exists 
$g\in\rmC(S,\R_+)$ such that $\{g\not =0\}=U$, for all $a\in\widetilde{S}$
\begin{equation*}
g\mycdot\Pbf_a\left(0<\xi<\infty\Rightarrow X_{\xi-}\text{ exists in }U\right)=1,
\end{equation*}
and the application
\begin{equation*}
\begin{array}{cccc}
g\mycdot\Pbf:&\widetilde{S}&\to&\Pcal(\{0<\xi<\infty\Rightarrow X_{\xi-}\text{ exists in }U\})\\
&a&\mapsto&g\mycdot\Pbf_a
\end{array}
\end{equation*}
is weakly continuous for the global Skorokhod topology from $\D(S^\Delta)$.
\end{proposition}

\section{Martingale local problem}
\subsection{Definition and first properties}
To begin with we recall the optional sampling theorem. 
Its proof can be found in Theorem 2.13 and Remark 2.14. p. 61 from \cite{EK86}.
\begin{theorem}[Optional sampling theorem]\label{thmOptSam}
Let $\left(\Omega,(\Gc_t)_{t\in\R_+},\P\right)$ be a filtered probability space and let $M$ be a \cadlag $(\Gc_t)_t$-martingale, then for all $(\Gc_{t+})_t$-stopping times $\tau$ and $\sigma$, with $\tau$ bounded,
\[
\E\left[M_\tau\mid\Gc_{\sigma+}\right]=M_{\tau\wedge\sigma},\quad\P\text{-almost surely}.
\]
In particular $M$ is a $(\Gc_{t+})_t$-martingale.
\end{theorem}
\begin{definition}[Martingale local problem]
Let $L$ be a subset of $\rmC_0(S)\times\rmC(S)$.
\begin{enumerate}
\item[a)] The set $\Mc(L)$ of solutions of the martingale local problem associated 
to $L$ is the set of $\Pbf\in\Pcal\left(\Dloc(S)\right)$ such that for all $(f,g)\in L$ and open subset $U\Subset S$:
\[
f(X_{t\wedge\tau^U})-\int_0^{t\wedge\tau^U}g(X_s)\d s\text{ is a }\Pbf\text{-martingale}
\]
with respect to the filtration $(\Fc_t)_t$ or, equivalent, to the filtration  $(\Fc_{t+})_t$. Recall that $\tau^U$ is given by \eqref{eqtauU}. The martingale \textit{local problem} should not be confused with the \textit{local martingale} problem (see Remark \ref{rkBLMP} below for a connection).
\item[b)] We say that there is existence of a solution for the martingale local problem if for any $a\in S$ there exists an element $\Pbf$ in $\Mc(L)$ such that $\Pbf(X_0=a)=1$.
\item[c)] We say that there is uniqueness of the solution for the martingale local problem if for any $a\in S$ there is at most one element $\Pbf$ in $\Mc(L)$ such that $\Pbf(X_0=a)=1$.
\item[d)] The martingale local problem is said well-posed if there is existence and uniqueness of the solution.
\end{enumerate}
\end{definition}
\begin{remark}\label{rkBLMP}
1) By using dominated convergence, for all $L\subset\rmC_0(S)\times\rmC(S)$, $(f,g)\in L\cap\rmC_0(S)\times\rmC_b(S)$ and $\Pbf\in\Mc(L)$, we have that
\[
f(X_t)-\int_0^{t\wedge\xi}g(X_s)\d s\text{ is a }\Pbf\text{-martingale}.
\]
Hence, if $L\subset\rmC_0(S)\times\rmC_b(S)$, the martingale local problem and the classical martingale problem are equivalent.\\
2) It can be proved that, for all $L\subset\rmC_0(S)\times\rmC(S)$, $(f,g)\in L$ and $\Pbf\in\Mc(L)$ such that \[\Pbf\big(\xi<\infty\mbox{ implies }\{X_s\}_{s<\xi}\Subset S\big)=1,\] 
we have
\[
f(X_t)-\int_0^{t\wedge\xi}g(X_s)\d s\text{ is a }\Pbf\text{-local martingale}.
\]
Indeed, it suffices to use the family of stopping times 
\[
\left\{\tau^U\vee\big(T\1_{\{\tau^U\leq T,\tau^U=\xi\}}\big)\mid U\Subset S,\,T\geq 0\right\},
\] 
to obtain the assertion.\\
3) We shall see that the uniqueness or, respectively, the existence of a solution for the martingale local problem when one starts from a fixed point implies the uniqueness or  the existence of a solution for the martingale local problem when one starts with an arbitrary measure (see Proposition \ref{propEUF} below).\\
4) Let $L\subset\rmC_0(S)\times\rmC(S)$ and $\Pbf\in\Mc(L)$ be. If $(f,g)\in L$ and  $U\Subset S$ is an open subset, then, by dominated convergence
\[
\frac{\Ebf\left[f(X_{t\wedge\tau^U})\mid\Fc_0\right]-f(X_0)}{t}=\Ebf\left[\frac{1}{t}\int_0^{t\wedge\tau^U}g(X_s)\d s\mid\Fc_0\right]\cv[\Pbf\text{-a.s.}]{t\to0}g(X_0).
\]
\par\vspace{-1.7\baselineskip}\qedhere
\end{remark}

Let us point out some useful properties concerning the martingale local problem:
\begin{proposition}[Martingale local problem properties]\label{propLMPP}
Let $L$ be a subset of  $\rmC_0(S)\times\rmC(S)$. 
\begin{enumerate}
\item\label{it1PropLMPP} \emph{(Time change)}
Take $h\in\rmC(S,\R_+)$ and denote
\begin{equation}\label{notationhL}
hL:=\left\{(f,hg)\mid (f,g)\in L\right\}.
\end{equation}
Then, for all $\Pbf\in\Mc(L)$,
\begin{equation}
h\mycdot \Pbf\in\Mc(hL).
\end{equation}
\item\label{it2PropLMPP} \emph{(Closer property)}
The closure with respect to $\rmC_0(S)\times\rmC(S)$ satisfies
\begin{equation}
\Mc\left(\overline{\vect (L)}\right)=\Mc(L).
\end{equation}
\item\label{it3PropLMPP} \emph{(Compactness and convexity property)}
Suppose that $\rmD(L)$ is a dense subset of $\rmC_0(S)$, where the domain 
of $L$ is defined by
\[
\rmD(L):=\left\{f\in\rmC_0(S)\mid \exists g\in\rmC(S),~(f,g)\in L\right\}.
\]
Then $\Mc(L)$ is a convex compact set for the local Skorokhod topology.
\item\label{it4PropLMPP} \emph{(Quasi-continuity)}
Suppose that $\rmD(L)$ is a dense subset of $\rmC_0(S)$, then for any $\Pbf\in\Mc(L)$, $\Pbf$ is $(\Fc_{t+})_t$-quasi-continuous. 
More precisely this means that for any $(\Fc_{t+})_t$-stopping times $\tau,\tau_1,\tau_2\ldots$
\begin{equation}
X_{\tau_n}\cv{n\to\infty} X_\tau\quad\Pbf\text{-almost surely on }\left\{\tau_n\cv{n\to\infty}\tau<\infty\right\},
\end{equation}
with the convention $X_\infty:=\Delta$. 

In particular for any $t\geq 0$, $\Pbf(X_{t-}=X_t)=1$,
\[
\Pbf\big(\Dloc(S)\cap\D(S^\Delta)\big)=\Pbf\big(\xi\in(0,\infty)\Rightarrow X_{\xi_-}\text{ exists in }S^\Delta\big)=1,
\]
and for any open subset $U\subset S$, $\Pbf(\tau^U<\infty\Rightarrow X_{\tau^U}\not \in U)=1$.
\end{enumerate}
\end{proposition}
The following result tell that the mapping $L\mapsto\Mc(L)$ is somehow upper semi-continuous.
\begin{proposition}\label{propUSCLMP}
Let $L_n,L\subset\rmC_0(S)\times\rmC(S)$ be such that 
\begin{equation}\label{cvgenerateur}
\forall (f,g)\in L,\quad\exists(f_n,g_n)\in L_n,\mbox{ such that }\quad f_n\cv[\rmC_0]{n\to\infty} f,~g_n\cv[\rmC]{n\to\infty}g.
\end{equation}
Then:
\begin{enumerate}
\item\label{it1PropUSCLMP} \emph{(Continuity)}
Let $\Pbf^n,\Pbf\in\Pcal\left(\Dloc(S)\right)$ be such that $\Pbf^n\in\Mc(L_n)$ and suppose that $\{\Pbf^n\}_n$ converges weakly to $\Pbf$ for the local Skorokhod 
topology. Then $\Pbf\in\Mc(L)$.
\item\label{it2PropUSCLMP} \emph{(Tightness)}
Suppose that $\rmD(L)$ is dense in $\rmC_0(S)$, then for any sequence $\Pbf^n\in\Mc(L_n)$,  $\{\Pbf^n\}_n$ is tight for the local Skorokhod topology.
\end{enumerate}
\end{proposition}

The proofs of the two  latter propositions are interlaced and will be developed  in the appendix (see \S \ref{preuves_entrelacees}). During these proofs we use the following result concerning the property of uniform continuity along stopping times of the martingale local problem.  Its proof is likewise postponed to the Appendix. 

\begin{lemma}\label{lmUQC}
Let $L_n,L\subset\rmC_0(S)\times\rmC(S)$ be such that $\rmD(L)$ is dense in $\rmC_0(S)$ and assume the convergence of the operators in the sense given by 
\eqref{cvgenerateur}.
Consider $\Kc$ a compact subset of $S$ and $\Uc$ an open subset of $S^2$ containing $\{(a,a)\}_{a\in S}$. For an arbitrary $(\Fc_{t+})_t$-stopping time $\tau_1$ we denote the $(\Fc_{t+})_t$-stopping time
\[
\tau(\tau_1):=\inf\left\{t\geq \tau_1\mid\{(X_{\tau_1},X_s)\}_{\tau_1\leq s\leq t}\not\Subset \Uc \right\}.
\]
Then for each $\eps>0$ there exist $n_0\in\N$ and $\delta>0$ such that: for any $n\geq n_0$, $(\Fc_{t+})_t$-stopping times $\tau_1\leq\tau_2$ and $\Pbf\in\Mc(L_n)$ satisfying $\Ebf[(\tau_2-\tau_1)\mathds{1}_{\{X_{\tau_1}\in \Kc\}}]\leq\delta$, we have
\[
\Pbf(X_{\tau_1}\in \Kc,~\tau(\tau_1)\leq \tau_2)\leq\eps,
\]
with the convention $X_\infty:=\Delta$.
\end{lemma}

\subsection{Existence and conditioning}

Before giving the result of existence of a solution for the martingale local problem, let us recall that $X^\tau_t=X_{\tau\wedge t}$ for $\tau$ a stopping time, and the classical 
positive maximal principle (see \cite{EK86}, p.165):
\begin{definition}
A subset $L\subset\rmC_0(S)\times\rmC(S)$ satisfies the positive maximum principle if for all $(f,g)\in L$ and $a_0\in S$ such that $f(a_0)=\sup_{a\in S}f(a)\geq 0$ then $g(a_0)\leq 0$.
\end{definition}
\begin{remark}\label{rqunivariate}
1) A linear subspace $L\subset\rmC_0(S)\times\rmC(S)$ satisfying the positive maximum principle
is univariate. Indeed for any $(f,g_1),(f,g_2)\in L$, applying the positive maximum principle to $(0,g_2-g_1)$ and $(0,g_1-g_2)$ we deduce that $g_1=g_2$.\\
2) Suppose furthermore that $\rmD(L)$ is dense in $\rmC_0(S)$, then as a consequence of the second part of Proposition \ref{propLMPP} and of Theorem \ref{thmExMP} below, the closure $\overline{L}$ in $\rmC_0(S)\times\rmC(S)$ satisfy the positive maximum principle, too.
\end{remark}

The existence of a solution for the martingale local problem result will be a consequence of Theorem 5.4 p. 199 from \cite{EK86}.
\begin{theorem}[Existence]\label{thmExMP}
Let $L$ be a linear  subspace of $\rmC_0(S)\times\rmC(S)$.
\begin{enumerate}
\item If there is existence of a solution for the martingale local problem associated to $L$, then $L$ satisfies the positive maximum principle.
\item Conversely, if $L$ satisfies the positive maximum principle and $D(L)$ is dense in $\rmC_0(S)$, then there is existence of a solution 
for the martingale local problem associated to $L$.
\end{enumerate}
\end{theorem}
\begin{proof}
Suppose that there is existence of a solution for the martingale local problem, let $(f,g)\in L$ and $a_0\in S$ be such that $f(a_0)=\sup_{a\in S}f(a)\geq 0$. If we take $\Pbf\in\Mc(L)$ such that $\Pbf(X_0=a_0)=1$, then, by the fourth part of Remark \ref{rkBLMP}  
\[
g(a_0)=\lim_{t\to 0}\frac{1}{t}(f(X_{t\wedge\tau^U})-f(a_0))\leq 0,
\]
so $L$ satisfies the positive maximum principle.

Let us prove the second part of Theorem \ref{thmExMP}.
Consider $\widetilde{L}_0$ a countable dense subset of $L$ and $L_0:=\vect(\widetilde{L}_0)$. There exists $h\in\rmC_0(S)$ such that for all $(f,g)\in\widetilde{L}_0$: $hg\in\rmC_0$, hence $\overline{L}=\overline{L_0}$ and $hL_0\subset\rmC_0(S)\times\rmC_0(S)$.
We apply Theorem 5.4 p. 199 of \cite{EK86}
to the univariate operator $hL_0$: for all  $a\in S$, there exists $\widetilde{\Pbf}\in\Pcal(\D(S^\Delta))$ such that $\widetilde{\Pbf}(X_0=a)=1$ and for all $(f,g)\in hL_0$
\[
f(X_t)-\int_0^tg(X_s)\d s\text{ is a }\widetilde{\Pbf}\text{-martingale}.
\]
Then $\Pbf:=\law_{\widetilde{\Pbf}}(X^{\tau^S})\in\Pcal\left(\Dloc(S)\cap\D(S^\Delta)\right)$, moreover for any $(f,g)\in hL_0$, for any open subset $U\Subset S$, for any $s_1\leq \cdots\leq s_k\leq s\leq t$ in $\R_+$ and for any $\varphi_1,\ldots,\varphi_k\in\rmC(S^\Delta)$,
\begin{multline*}
\Ebf\left[\left(f(X_{t\wedge\tau^U})-f(X_{s\wedge\tau^U})-\int_{s\wedge\tau^U}^{t\wedge\tau^U}g(X_u)\d u\right)\varphi_1(X_{s_1})\cdots \varphi_k(X_{s_k})\right]\\
\widetilde{\Ebf}\left[\left(f(X_{t\wedge\tau^U})-f(X_{s\wedge\tau^U})-\int_{s\wedge\tau^U}^{t\wedge\tau^U}g(X_u)\d u\right)\varphi_1(X_{s_1})\cdots \varphi_k(X_{s_k})\right] = 0.
\end{multline*}
Hence $\Pbf\in\Mc(hL_0)$.  To conclude we use the two first 
parts of Proposition \ref{propLMPP}:
\[
\Mc(L)=\Mc(\overline{L})=\Mc(L_0)=\left\{\frac{1}{h}\mycdot\Qbf\mid\Qbf\in\Mc(hL_0)\right\}.
\]
So $\frac{1}{h}\mycdot \Pbf\in\Mc(L)$ and the existence of a solution for the martingale local problem is proved.
\end{proof}

\begin{remark}\label{rkExCL}
Since $\Fc$ is the Borel $\sigma$-algebra on the Polish space $\Dloc(S)$,
we can use Theorem 6.3, in \cite{Ka02},  p. 107. So, for any $\Pbf\in\Pcal\left(\Dloc(S)\right)$ and $(\Fc_{t+})_t$-stopping time $\tau$, the regular conditional distribution $\Qbf_X\overset{\Pbf\text{-a.s.}}{:=}\law_\Pbf\big((X_{\tau+t})_{t\geq 0}\,\big|\,\Fc_{\tau+}\big)$ exists. It means that there exists
\[\begin{array}{cccc}
\Qbf:&\Dloc(S)&\to&\Pcal\left(\Dloc(S)\right)\\
&x&\mapsto&\Qbf_x
\end{array}\] 
such that for any $A\in\Fc$, $\Qbf_X(A)$ is $\Fc_{\tau+}$-measurable and
\[
\Pbf\big((X_{\tau+t})_{t\geq 0}\in A\,\big|\,\Fc_{\tau+}\big)=\Qbf_X(A)\quad\Pbf\text{-almost surely}.\qedhere
\]
\end{remark}
\begin{proposition}[Conditioning]\label{propCondLMP}
Take $L\subset\rmC_0(S)\times\rmC(S)$, $\Pbf\in\Mc(L)$, and a $(\Fc_{t+})_t$-stopping time $\tau$. As in Remark \ref{rkExCL} we denote $\Qbf_X\overset{\Pbf\text{-a.s.}}{:=}\law_\Pbf\big((X_{\tau+t})_{t\geq 0}\,\big|\,\Fc_{\tau+}\big)$, then
\[
\Qbf_X\in\Mc(L),\quad\Pbf\text{-almost surely}.
\]
\end{proposition}
\begin{proof}
Let $(f,g)$ be in $L$, $s_1\leq \cdots\leq s_k\leq s\leq t$ be in $\R_+$, $\varphi_1,\ldots,\varphi_k$ be in $\rmC(S^\Delta)$ and $U\Subset S$ be a open subset. Here and elsewhere we will denote by $E^{\Qbf_x}$ the expectation with respect to $\Qbf_x$. Since
\begin{multline*}
\mathds{1}_{\tau<\tau^U}E^{\Qbf_X}\Big[\big(f(X_{t\wedge\tau^U})-f(X_{s\wedge\tau^U})-\int_{s\wedge\tau^U}^{t\wedge\tau^U}g(X_u)\d u\big)\varphi_1(X_{s_1})\cdots \varphi_k(X_{s_k})\Big]\\
\overset{\Pbf\text{-a.s.}}{=}\mathds{1}_{\tau<\tau^U}\Ebf\Big[\big(f(X_{(t+\tau)\wedge\tau^U})-f(X_{(s+\tau)\wedge\tau^U})-\int_{(s+\tau)\wedge\tau^U}^{(t+\tau)\wedge\tau^U}g(X_u)\d u\big)\\
\times\varphi_1(X_{s_1+\tau})\cdots \varphi_k(X_{s_k+\tau})\,\Big|\,\Fc_{\tau+}\Big]
\overset{\Pbf\text{-a.s.}}{=}0,
\end{multline*}
we have
\begin{multline}\label{eqPropCondLMP}
\Ebf\Big(E^{\Qbf_X}\Big[\big(f(X_{t\wedge\tau^U})-f(X_{s\wedge\tau^U})-\int_{s\wedge\tau^U}^{t\wedge\tau^U}g(X_u)\d u\big)\varphi_1(X_{s_1})\cdots \varphi_k(X_{s_k})\Big]\not = 0\Big)\\
\leq\Pbf\big(\tau^U\leq\tau<\xi\big).
\end{multline}
Let $\widetilde{L}$ be a countable dense subset of $L$, $C$ be a countable dense subset of $\rmC(S^\Delta)$ and $U_n\Subset S$ be an increasing sequence of open subsets such that $S=\bigcup_nU_n$.
Then $\Qbf_X\in\Mc(L)$ if and only if for all $(f,g)\in \widetilde{L}$, $k\in\N$, for any $s_1\leq \cdots\leq s_k\leq s\leq t$ in $\Q_+$, for any $\varphi_1,\ldots,\varphi_k\in C$, and for $n$ large enough
\[
E^{\Qbf_X}\Big[\big(f(X_{t\wedge\tau^{U_n}})-f(X_{s\wedge\tau^{U_n}})-\int_{s\wedge\tau^{U_n}}^{t\wedge\tau^{U_n}}g(X_u)\d u\big)\varphi_1(X_{s_1})\cdots \varphi_k(X_{s_k})\Big]=0.
\]
Hence $\{\Qbf_X\in\Mc(L)\}$ is in $\Fc_{\tau+}$ and by \eqref{eqPropCondLMP}, $\Pbf$-almost surely $\Qbf_X\in\Mc(L)$.
\end{proof}
\begin{proposition}\label{propEUF}
Set $L\subset\rmC_0(S)\times\rmC(S)$. 
\begin{enumerate}
\item If there is uniqueness of the solution for the martingale local problem then for any $\mu\in\Pcal(S^\Delta)$ there is at most one element $\Pbf$ in $\Mc(L)$ such that $\law_\Pbf(X_0) =\mu$.
\item If there is existence of a solution for the martingale local problem and $\rmD(L)$ is dense in $\rmC_0(S)$, then for any $\mu\in\Pcal(S^\Delta)$ there exists an element $\Pbf$ in $\Mc(L)$ such that $\law_\Pbf(X_0) =\mu$.
\end{enumerate}
\end{proposition}
\begin{proof}
Suppose that we have uniqueness of the solution for the martingale local problem. Let $\mu$ be in $\Pcal(S^\Delta)$ and $\Pbf^1,\Pbf^2\in\Mc(L)$ be such that $\law_{\Pbf^1}(X_0)=\law_{\Pbf^2}(X_0) =\mu$. As in Remark \ref{rkExCL} let $\Qbf_\bullet,\Rbf_\bullet:S^\Delta\to\Pcal(\Dloc(S))$ be such that 
\[
\Qbf_{X_0}\overset{\Pbf^1\text{-a.s.}}{:=}\law_{\Pbf^1}\left(X\mid\Fc_0\right),\hspace{2cm}\Rbf_{X_0}\overset{\Pbf^2\text{-a.s.}}{:=}\law_{\Pbf^2}\left(X\mid\Fc_0\right).
\]
Then, by Proposition \ref{propCondLMP}, $\Qbf_{a},\Rbf_{a}\in\Mc(L)$ for $\mu$-almost all $a$, 
so, by uniqueness of the solution for the martingale local problem, $\Qbf_{a}=\Rbf_{a}$ for $\mu$-almost all $a$. We finally obtain $\Pbf^1=\int\Qbf_a\nu(\d a)=\int\Rbf_a\nu(\d a)=\Pbf^2$.

Suppose that we have existence of a solution for the martingale local problem and that $\rmD(L)$ is dense in $\rmC_0(S)$. Thanks to \ref{it3PropLMPP} from Proposition \ref{propLMPP} $\Mc(L)$ is convex and compact. Hence the set 
\[
C:=\{\mu\in\Pcal(S^\Delta)\,|\,\exists\Pbf\in\Mc(L)\,\mbox{ such that }\,\law_\Pbf(X_0) =\mu\}
\]
is convex and compact. Since there is existence of a solution for the martingale local problem we have $\left\{\delta_a\mid a\in S^\Delta\right\}\subset C$ so $C=\Pcal(S^\Delta)$.
\end{proof}

\section{Locally Feller families of probabilities}

In this section we will study a local counterpart of Feller families in connection 
with Feller semi-groups and martingale local problems. The basic notions and 
facts on Feller semi-groups can be founded in Chapter 19 pp. 367-389 from \cite{Ka02}.

\subsection{Feller families of probabilities}

Let $(\Gc_t)_{t\geq 0}$ be a filtration containing $(\Fc_t)_{t\geq 0}$. Recall that a family of probability measures 
$(\Pbf_a)_{a\in S}\in\Pcal(\Dloc(S))^S$ is called $(\Gc_t)_t$-Markov if, for any $B\in\Fc$, $a\mapsto\Pbf_a(B)$ 
is measurable, for any $a\in S$,  $\Pbf_a(X_0=a)=1$, and for any $B\in\Fc$, $a\in S$ and $t_0\in\R_+$
\[
\Pbf_a\left((X_{t_0+t})_t\in B\mid\Gc_{t_0}\right)=\Pbf_{X_{t_0}}(B),\;\Pbf_a-\text{almost surely},
\]
where $\Pbf_\Delta$ is the unique element of $\Pcal(\Dloc(S))$ such that $\Pbf_\Delta(\xi=0)=1$.
If the latter property is also satisfied by replacing $t_0$ with any $(\Gc_t)_t$-stopping time,
the family of probability measures  is $(\Gc_t)_t$-strong Markov. If $\Gc_t=\Fc_t$ we just say that the family 
is (strong) Markov. If $\nu$ is a measure on $S^\Delta$ we set $\Pbf_\nu:=\int\Pbf_a\nu(\d a)$. Then the 
distribution of $X_0$ under $\Pbf_\nu$ is $\nu$, and $\Pbf_\nu$ satisfies the (strong) Markov property.

\begin{definition}[Feller family]
A Markov family $(\Pbf_a)_a\in\Pcal(\Dloc(S))^S$ is said to be Feller if for all $f\in\rmC_0(S)$ and $t\in\R_+$ the function
\[\begin{array}{cccc}
T_tf:&S&\to&\R\\
&a&\mapsto&\Ebf_a[f(X_t)]
\end{array}\]
is in $\rmC_0(S)$. In this case it is no difficult to see that $(T_t)_t$ is a Feller 
semi-group on $\rmC_0(S)$ (see p. 369 in \cite{Ka02}) called the semi-group of 
$(\Pbf_a)_a$. 
Its generator $L$ is is the set of $(f,g)\in\rmC_0(S)\times\rmC_0(S)$ such that, for all $a\in S$
\[
\frac{T_tf(a)-f(a)}{t}\cv{t\to 0}g(a).
\]
and we call it the $\rmC_0\times\rmC_0$-generator of $(\Pbf_a)_a$.
\end{definition}
In \cite{vC92} Theorem 2.5, p. 283, one states a connection between Feller families 
and martingale problems. Unfortunately the proof given in the cited paper is correct 
only on a compact space $S$. The fact that a Feller family of probabilities is the unique 
solution of an appropriate martingale problem is stated in the proposition below. We will 
prove the converse of this result in Theorem \ref{FirstcharFeller}.

To give this statement we need to introduce some notations. For $L\subset\rmC_0(S)\times\rmC_0(S)$ we define
\begin{align}\label{eqLDlt}
L^\Delta:=\vect\left(L\cup\{(\mathds{1}_{S^\Delta},0)\}\right)\subset\rmC(S^\Delta)\times\rmC(S^\Delta).
\end{align}
We recall that we identified $\rmC_0(S)$ by the set of functions $f\in\rmC(S)$ such that $f(\Delta)=0$.
The set of solutions $\Mc(L^\Delta)\subset\Pcal(\Dloc(S^\Delta))$ of the martingale problem associated to $L^\Delta$ satisfies
\[
\forall\Pbf\in\Mc(L^\Delta),\quad\Pbf(X_0\in S^\Delta\Rightarrow X\in\D(S^\Delta))=0.
\]
Without loss of the generality, to study the martingale problem associated to $L^\Delta$ 
it suffices to  study the set of solution with $S^\Delta$-conservative paths:
\[
\Mcons(L^\Delta):=\Mc(L^\Delta)\cap\Pcal(\D(S^\Delta))
=\left\{\Pbf\in\Mc(L^\Delta)\mid\Pbf(X_0\in S^\Delta)=1\right\}.
\]
In fact $\Mcons(L^\Delta)$ is the set consisting of $\Pbf\in\Pcal(\D(S^\Delta))$ such that for all $(f,g)\in L$
\begin{align}\label{eqMgDlt}
f(X_t)-\int_0^tg(X_s)\d s\quad\text{ is a }\Pbf\text{-martingale}.
\end{align}
\begin{proposition}\label{propUniMgPb}
If $(T_t)_t$ is a Feller semi-group on $\rmC_0(S)$ with $L$ its generator, then there is
a unique Feller family $(\Pbf_a)_a$ with semi-group $(T_t)_t$. Moreover the martingale problem associate to $L^\Delta$ is well-posed and
\[
\Mcons(L^\Delta)=\{\Pbf_\mu\}_{\mu\in\Pcal(S^\Delta)}.
\]
\end{proposition}
\begin{remark}\label{rkMPinCompactified}
1. For any $\Pbf\in\Mcons(L^\Delta)$ the distribution of $X^{\tau^S}$ under $\Pbf$ satisfies
\[
\law_\Pbf(X^{\tau^S})\in\Mcons(L^\Delta)\cap\Dloc(S)\subset\Mc(L).
\]
Moreover if $\rmD(L)$ is dense in $\rmC_0(S)$, thanks to \ref{it4PropLMPP} from Proposition \ref{propLMPP}
\begin{align*}
\Mc(L)=\Mcons(L^\Delta)\cap\Dloc(S).
\end{align*}
So if $\rmD(L)$ is dense in $\rmC_0(S)$ there is existence of a solution for the martingale problem associated to $L$ if and only if there is existence of a solution to the martingale problem associated to $L^\Delta$. Moreover the uniqueness of the solution for the martingale problem associated to $L^\Delta$ imply uniqueness of the solution for the martingale problem associated to $L$.\\
2. If $S$ is compact and $\rmD(L)$ is dense in $\rmC_0(S)=\rmC(S)$, then it is straightforward to obtain $\Mc(L)=\Mcons(L^\Delta)$.
\end{remark}

\noindent
For the sake of completeness we give: 
\begin{proof}[Proof of Proposition \ref{propUniMgPb}]
The existence of a solution for the martingale problem is a consequence of Theorem \ref{thmExMP}.  
Thanks to Proposition \ref{propCondLMP}, to prove our result we need to prove that 
\[
\forall\Pbf\in\Mcons(L^\Delta),\,\forall t\geq 0,\,\forall f\in\rmD(L),\;
\Ebf\left[f(X_t)\right]=\Ebf\left[T_tf(X_0)\right]
\]
Let $0=t_0\leq\cdots\leq t_{N+1}=t$ be a subdivision of $[0,t]$, then
\begin{align*}
\Ebf\left[f(X_t)\mid\Fc_0\right]-T_tf(X_0)
& =\sum_{i=0}^N\Ebf\left[T_{t-t_{i+1}}f(X_{t_{i+1}})\mid\Fc_0\right]-\Ebf\left[T_{t-t_i}f(X_{t_i})\mid\Fc_0\right] \\
& =\sum_{i=0}^N\Ebf\left[\Ebf\left[T_{t-t_{i+1}}f(X_{t_{i+1}})\mid\Fc_{t_i}\right]-T_{t-t_i}f(X_{t_i})\mid\Fc_0\right].
\end{align*}
Moreover for each $i\in\{0,\ldots N\}$, using martingales properties for the first part and semi-groups properties (see for instance Theorem 19.6, p. 372 in \cite{Ka02}) for the second
\begin{equation*}
\Ebf\Big[T_{t-t_{i+1}}f(X_{t_{i+1}})\,\big|\,\Fc_{t_i}\Big]-T_{t-t_i}f(X_{t_i})
=\Ebf\Big[\int_{t_i}^{t_{i+1}}LT_{t-t_{i+1}}f(X_s)-LT_{t-s}f(X_{t_i})\d s\,\big|\,\Fc_{t_i}\Big],
\end{equation*}
so
\begin{align*}
\left|\Ebf\left[f(X_t)-T_tf(X_0)\right]\right|
& \leq \Ebf\sum_{i=0}^N\int_{t_i}^{t_{i+1}}\left|LT_{t-t_{i+1}}f(X_s)-LT_{t-s}f(X_{t_i})\right|\d s .
\end{align*}
By dominated convergence we can conclude.
\end{proof}
Before introducing the definition of a locally Feller family, let us state a result 
on an application of a time change to a Feller family:
\begin{proposition}\label{propTCFel}
Let $(\Pbf_a)_a\in\Pcal(\Dloc(S))^S$ be a Feller family with $\rmC_0\times\rmC_0$-generator $L$. Then, for any $g\in\rmC_b(S,\R_+^*)$, $(g\mycdot\Pbf_a)_a$ is a Feller family with $\rmC_0\times\rmC_0$-generator $\overline{gL}$, taking the closure in $\rmC_0(S)\times\rmC_0(S)$.
\end{proposition}
\begin{proof}
Thanks to the first part of Proposition \ref{propLMPP} and to the Proposition \ref{propUniMgPb}, the result is only a reformulation of Theorem 2, p. 275 in \cite{Lu73}. For the sake of completeness we give the statement of this result in our context: if $L\subset\rmC_0(S)\times\rmC_0(S)$ is the generator of a Feller semi-group, then for any $g\in\rmC_b(S,\R_+^*)$, $\overline{gL}$ is the generator of a Feller semi-group. 
\end{proof}

\subsection{Local Feller families and connection with martingale problems}

We are ready to introduce the notion of locally Feller family of probabilities. This is given in 
the following theorem which proof  is technical and it is postponed to the Appendix \S
\ref{secProofslocallyFeller}
\begin{theorem}[Definition of a locally Feller family]\label{thmDefLFF}
If $(\Pbf_a)_a\in\Pcal(\Dloc(S))^S$, the following four assertions are equivalent:
\begin{enumerate}
\item\label{item1thmDefLFF} (continuity) the family $(\Pbf_a)_a$ is Markov and $a\mapsto\Pbf_a$ is continuous for the local Skorokhod topology;
\item\label{item2thmDefLFF} (time change) there exists $g\in\rmC(S,\R_+^*)$ such that $(g\mycdot\Pbf_a)_a$ is a Feller family;
\item\label{item3thmDefLFF} (martingale) there exists $L\subset\rmC_0(S)\times\rmC(S)$ such that $\rmD(L)$ is dense in $\rmC_0(S)$ and
\[
\forall a\in S,\quad
\Pbf\in\Mc(L)\,\mbox{ and }\,\Pbf(X_0=a)=1
\Longleftrightarrow
\Pbf=\Pbf_a;
\]
\item\label{item4thmDefLFF} (localisation) for any open subset $U\Subset S$ there 
exists a Feller family $(\widetilde{\Pbf}_a)_a$ such that for any $a\in S$
\[
\law_{\Pbf_a}\left(X^{\tau^U}\right) = \law_{\widetilde{\Pbf}_a}\left(X^{\tau^U}\right).
\]
\end{enumerate}
We will call a such family a locally Feller family. 

\noindent
Moreover a locally Feller family $(\Pbf_a)_a$ is $(\Fc_{t+})_t$-strong Markov and for all $\mu \in \Pcal(S^\Delta)$, $\Pbf_\mu$ is quasi-continuous.
\end{theorem}
\begin{remark}\label{construirelFeller}
A natural question is how can we construct locally Feller families? We 
give here answers to this question.
\begin{itemize}
\item[i)]  A Feller family is locally Feller.
\item[ii)] If $g\in\rmC(S,\R_+^*)$  and $(\Pbf_a)_a\in\Pcal(\Dloc(S))^S$ is locally Feller, then $(g\mycdot\Pbf_a)_a$ is locally Feller. This result is to be compared with the result 
of Proposition \ref{propTCFel}.
\item[iii)] If $S$ is a compact space, a family is locally Feller if and only if it is Feller. This sentence is an easy consequence of the third part of the latter theorem and of Proposition 
\ref{propTCFel}.
\item[iv)]  As consequence of the first assertion in Theorem \ref{thmDefLFF},
if $(\Pbf_a)_a\in\Pcal(\Dloc(S))^S$ is locally Feller then the family
\[\begin{array}{ccc}
U&\to&\Pcal(\Dloc(U))\\
a&\mapsto&\law_{\Pbf_a}(\widetilde{X})
\end{array}\]
is locally Feller in the space $U$. Indeed, it is straightforward to verify that, for any open subset $U\subset S$, the following mapping is continuous,
\[
\begin{array}{ccc}
\Dloc(S)&\to&\Dloc(U)\\
x&\mapsto&\widetilde{x}
\end{array}
\quad\text{ with }\quad
\widetilde{x}_s:=\left\{\begin{array}{ll}
x_s&\text{if }s<\tau^U(x),\\
\Delta &\text{otherwise.}
\end{array}\right.
\]
\end{itemize}
\par\vspace{-1.7\baselineskip}\qedhere
\end{remark}


Since a locally  Feller family on $S^\Delta$ is also Feller we can deduce from Theorem 
\ref{thmDefLFF} a characterisation of Feller families in terms of martingale problem. The following 
theorem is the converse of Proposition \ref{propUniMgPb} and provide a first correction of the 
result Theorem 2.5,  p. 283 in \cite{vC92}.
\begin{theorem}[Feller families - first characterisation]\label{FirstcharFeller}
Let $(\Pbf_a)_a\in\Pcal(\Dloc(S))^S$ be, the following assertions are equivalent:
\begin{enumerate}
\item\label{propEqFlFen1} $(\Pbf_a)_a$ is Feller;
\item\label{propEqFlFen2} the family $(\Pbf_a)_a$ is Markov, $\Pbf_a\in\Pcal(\D(S^\Delta))$ for any $a\in S$, and $S^\Delta\ni a\mapsto\Pbf_a$ is continuous for the global Skorokhod topology;
\item\label{propEqFlFen3} there exists $L\subset\rmC_0(S)\times\rmC_0(S)$ such that $\rmD(L)$ is dense in $\rmC_0(S)$ and
\[
\forall a\in S^\Delta,\quad
\Pbf\in\Mcons(L^\Delta)\,\mbox{ and }\,\Pbf(X_0=a)=1
\Longleftrightarrow
\Pbf=\Pbf_a.
\]
We recall that $\Pbf_\Delta$ is defined by $\Pbf_\Delta(\forall t\geq 0,~X_t=\Delta)=1$.
\end{enumerate}
\end{theorem}
\begin{proof}
Thanks to the fourth point of Proposition \ref{propLMPP} a Feller family in $\Pcal(\Dloc(S))$ continues to be Feller also in $\Pcal(\D(S^\Delta))$, so a family $(\Pbf_a)_a\in\Pcal(\Dloc(S))^S$ is Feller if and only if the family $(\Pbf_a)_a\in\Pcal(\D(S^\Delta))^{S^\Delta}$ is Feller. Since $S^\Delta$ is compact, using the third point of Remark \ref{construirelFeller}, this is also equivalent to say that $(\Pbf_a)_{a\in S^\Delta}$ is locally Feller in $S^\Delta$. Hence the theorem is a consequence of Theorem \ref{thmDefLFF} applied on the space $S^\Delta$ and to Proposition \ref{propUniMgPb}.
\end{proof}

The following theorem provides a new relationship between the local Feller property and the Feller property. With the help of Theorem \ref{thmDefLFF} we obtain another correction of the Theorem 2.5 p. 283 from \cite{vC92} by adding the missing condition \eqref{conditionsupplementairevC}.

\begin{theorem}[Feller families - second characterisation]\label{SecondCharFeller}
Let $(\Pbf_a)_a\in\Pcal(\Dloc(S))^S$ be, the following assertions are equivalent:
\begin{enumerate}
\item $(\Pbf_a)_a$ is Feller;
\item\label{propEqFlFen4} $(\Pbf_a)_a$ is locally Feller and
\begin{equation}\label{conditionsupplementairevC}
\forall t\geq 0,~\forall K\subset S\mbox{ compact set,}\quad\Pbf_a(X_t\in K)\cv{a\to\Delta}0;
\end{equation}
\item\label{propEqFlFen5} $(\Pbf_a)_a$ is locally Feller and 
\[
\forall t\geq 0,~\forall K\subset S\mbox{ compact set,}\quad\Pbf_a\big(\tau^{S\backslash K}< t\wedge\xi\big)\cv{a\to\Delta}0.
\]
\end{enumerate}
\end{theorem}
\begin{proof}
\emph{\ref{propEqFlFen1}$\Rightarrow$\ref{propEqFlFen4}.}
Take a compact $K\subset S$ and $t\geq 0$. There exists $f\in\rmC_0(S)$ such that $f\geq \1_K$. Since the family is Feller,
\[
\Pbf_a(X_t\in K)\leq\Ebf_a[f(X_t)]\cv{a\to\Delta}0.
\]

\noindent
\emph{\ref{propEqFlFen4}$\Rightarrow$\ref{propEqFlFen5}.}
Take an open subset $U\Subset S$ such that $K\subset U$ and define
\[
\tau:=\inf\Big\{s\geq 0\,\Big|\,\{(X_0,X_u)\}_{0\leq u\leq s}\not\Subset U^2\cup(S\backslash K)^2 \Big\}.
\]
By the third sentence of Theorem \ref{thmDefLFF}, we can applying Lemma \ref{lmUQC} to $\Kc:=K$, $\Uc:=U^2\cup(S\backslash K)^2$, $\tau_1:=0$ and $\tau_2:=\frac{t}{N}$, we get the existence of $N\in\N$ such that
\[
\sup_{b\in K}\Pbf_b\Big(\tau\leq\frac{t}{N}\Big)<1.
\]
By Theorem \ref{thmDefLFF}, $\Pbf_a$ is quasi-continuous for any $a\in S$, so $\Pbf_a(X_{\tau^{S\backslash K}}\in K\cup\{\Delta\})=1$. Denoting $\lceil r\rceil$ the smallest integer larger or equal than the real number $r$, we have
\begin{multline*}
\Pbf_a\Big(\exists k\in \N,~k\leq N,~X_{ktN^{-1}}\in U\Big)
\geq\Pbf_a\Big(\tau^{S\backslash K}< t\wedge\xi,~X_{tN^{-1}\lceil t^{-1}N\tau^{S\backslash K}\rceil}\in U\Big)\\
=\Ebf_a\Big[\mathds{1}_{\{\tau^{S\backslash K}< t\wedge\xi\}}\Ebf_{X_{\tau^{S\backslash K}}}\left[X_s\in U\right]_{|s=tN^{-1}\lceil t^{-1}N\tau^{S\backslash K}\rceil-\tau^{S\backslash K}}\Big]\\
\geq\Pbf_a\Big(\tau^{S\backslash K}< t\wedge\xi\Big)\Big[1-\sup_{b\in K}\Pbf\big(\tau\leq tN^{-1}\big)\Big],
\end{multline*}
so
\[
\Pbf_a\big(\tau^{S\backslash K}< t\wedge\xi\big)\leq\frac{\sum_{k=0}^N\Pbf_a\big(X_{ktN^{-1}}\in U\big)}{1-\sup_{b\in K}\Pbf_b\big(\tau\leq tN^{-1}\big)}\cv{}0,\quad\text{as }a\to\Delta.
\]
\emph{\ref{propEqFlFen5}$\Rightarrow$\ref{propEqFlFen1}.}
Consider  $f\in\rmC_0(S)$ and let $t\geq 0$ and $\eps>0$ be.  There exists a compact subset $K\subset S$ such that
$\|f\|_{K^c} \leq \eps$,
and an open subset $U\Subset S$ such that $K\subset U$ and
\[
\sup_{a\not\in U}\Pbf_a(\tau^{S\backslash K}< t\wedge\xi)\leq \eps.
\]
With the aim of the second assertion of Theorem \ref{thmDefLFF} and Proposition \ref{propTCFel}, there exists $g\in\rmC(S,(0,1])$ such that $g(a) = 1$, for $a\in U$, and $(g\mycdot\Pbf_a)_a$ is Feller. Then for any $a\in S$
\begin{align*}
\big|\Ebf_a[f(X_t)]-\Ebf_a[f((g\mycdot X)_t)]\big|
& \leq \Ebf_a\big[\left|f(X_t)-f((g\mycdot X)_t)\right|\mathds{1}_{\{\tau^U< t\}}\big] \\
& \leq \Ebf_a\big[\left|f(X_t)\right|\mathds{1}_{\{\tau^U< t\}}\big]
 +  \Ebf_a\big[\left|f((g\mycdot X)_t)\right|\mathds{1}_{\{\tau^U< t\}}\big].
\end{align*}
By Theorem \ref{thmDefLFF}, $\Pbf_a$ is quasi-continuous, so $\Pbf_a(X_{\tau^U}\not\in U)=1$, we have
\begin{align*}
\Ebf_a\Big[\big|f(X_t)\big|\mathds{1}_{\{\tau^U< t\}}\Big]
& = \Ebf_a\Big[\mathds{1}_{\{\tau^U< t\}}\Ebf_{X_{\tau^U}}\big[|f(X_s)|\big]_{|s=t-\tau^U}\Big] \\
& = \Ebf_a\Big[\mathds{1}_{\{\tau^U< t\}}\Ebf_{X_{\tau^U}}\big[|f(X_s)|\1_{\{\tau^{S\backslash K}< t\wedge\xi\}}\big]_{|s=t-\tau^U}\Big] \\
& \quad + \Ebf_a\Big[\mathds{1}_{\{\tau^U< t\}}\Ebf_{X_{\tau^U}}\big[|f(X_s)|\mathds{1}_{\{\tau^{S\backslash K}\geq t\wedge\xi\}}\big]_{|s=t-\tau^U}\Big] \\
& \leq \|f\|\sup_{a\not\in U}\Pbf_a(\tau^{S\backslash K}< t\wedge\xi) + \|f\|_{K^c}\leq (\|f\|+1)\eps,
\end{align*}
and
\begin{align*}
\Ebf_a\Big[\left|f(g\mycdot X_t)\right|\mathds{1}_{\{\tau^U< t\}}\Big]
& = \Ebf_a\Big[\mathds{1}_{\{\tau^U< t\}}\Ebf_{X_{\tau^U}}\big[|f(g\mycdot X_s)|\big]_{|s=t-\tau^U}\Big] \\
& = \Ebf_a\Big[\mathds{1}_{\{\tau^U< t\}}\Ebf_{X_{\tau^U}}\big[|f(g\mycdot X_s)|\1_{\{\tau^{S\backslash K}< t\wedge\xi\}}\big]_{|s=t-\tau^U}\Big] \\
& \quad + \Ebf_a\Big[\mathds{1}_{\{\tau^U< t\}}\Ebf_{X_{\tau^U}}\big[|f(g\mycdot X_s)|\1_{\{\tau^{S\backslash K}\geq t\wedge\xi\}}\big]_{|s=t-\tau^U}\Big] \\
& \leq \|f\|\sup_{a\not\in U}\Pbf_a(\tau^{S\backslash K}< t\wedge\xi) + \|f\|_{K^c}\leq (\|f\|+1)\eps.
\end{align*}
Hence
\[
\big|\Ebf_a[f(X_t)]-\Ebf_a[f((g\mycdot X)_t)]\big|\leq 2(\|f\|+1)\eps,
\]
so, since $a\mapsto \Ebf_a[f((g\mycdot X)_t)]$ is in $\rmC_0(S)$, letting $\eps\to 0$ we deduce that $a\mapsto \Ebf_a[f(X_t)]$ is in $\rmC_0(S)$, hence $(\Pbf_a)_a$ is Feller.
\end{proof}

\subsection{Generator description and convergence}

In this subsection we analyse the generator of a locally Feller family: 

\begin{definition}
Let $(\Pbf_a)_a\in\Pcal(\Dloc(S))^S$ be a locally Feller family. The $\rmC_0\times\rmC$-generator $L$ of $(\Pbf_a)_a\in\Pcal(\Dloc(S))^S$ is the set of functions $(f,g)\in\rmC_0(S)\times\rmC(S)$ such that for any $a\in S$ and any 
open subset $U\Subset S$
\[
f(X_{t\wedge\tau^U})-\int_0^{t\wedge\tau^U}g(X_s)\d s\text{ is a }\Pbf_a\text{-martingale}.
\]
\end{definition}

\begin{theorem}[Generator's description]\label{thmGenDesc}
Let $(\Pbf_a)_a\in\Pcal(\Dloc(S))^S$ be a locally Feller family and $L$ its $\rmC_0\times\rmC$-generator. Then $\rmD(L)$ is dense, $L$ is an 
univariate closed sub-vector space,
\[
\Mc(L)=\{\Pbf_\mu\}_{\mu\in \Pcal(S^\Delta)},
\]
$L$ satisfies the positive maximum principle and does not have a strict linear extension satisfying the positive maximum principle.
Moreover for any $(f,g)\in\rmC_0(S)\times\rmC(S)$ we have equivalence between:
\begin{enumerate}
\item\label{eq1thmGen} $(f,g)\in L$;
\item\label{eq2thmGen} for all $a\in S$, there exists an open set $U\subset S$ containing $a$ such that
\[
\lim_{t\to 0}\frac{1}{t}\Big(\Ebf_a\left[f(X_{t\wedge\tau^U})\right]-f(a)\Big)=g(a);
\]
\item\label{eq3thmGen} for all open subset $U\Subset S$ and $a\in U$
\[
\lim_{t\to 0}\frac{1}{t}\Big(\Ebf_a\left[f(X_{t\wedge\tau^U})\right]-f(a)\Big)=g(a).
\]
\end{enumerate}
\end{theorem}
\begin{proof}
Thanks to the third assertion of Theorem \ref{thmDefLFF} and Proposition \ref{propEUF}, we have $\Mc(L)=\{\Pbf_\nu\}_{\nu\in \Pcal(S^\Delta)}$ and $\rmD(L)$ is dense. By  the point \ref{it2PropLMPP} of Proposition \ref{propLMPP}, $L$ is a closed sub-vector space.
The fourth part of Remark \ref{rkBLMP} allows us to conclude that: $L$ is univariate, $L$ satisfies the positive maximum principle, and that \ref{eq1thmGen}$\Rightarrow$\ref{eq3thmGen}. It is strightforward that \ref{eq3thmGen}$\Rightarrow$\ref{eq2thmGen}. Thanks to Theorem \ref{thmExMP}, $L$ does not have strict linear extension satisfying the positive maximum principle. Finally the set of $(f,g)$ satisfying the statement \ref{eq2thmGen} is a linear extension of $L$ satisfying the positive maximum principle, so by the previous assertion \ref{eq2thmGen}$\Rightarrow$\ref{eq1thmGen}.
\end{proof}

\begin{remark}\label{rkTCGen}
One can ask, as in Remark \ref{construirelFeller}, how can we obtain the generator of a 
locally Feller family? A similar statement of first one in the cited remark is 
Proposition \ref{propGenFF} below. The second one is straightforward: if  $g\in\rmC(S,\R_+^*)$ and if $L$ is the $\rmC_0\times\rmC$-generator of $(\Pbf_a)_a$, then $gL$ is the $\rmC_0\times\rmC$-generator of $(g\mycdot\Pbf_a)_a$, as we can see by using \ref{it1PropLMPP} from Proposition \ref{propLMPP}.
\end{remark}

\begin{proposition}
\label{propGenFF}
Let $(\Pbf_a)_a\in\Pcal(\Dloc(S))^S$ be a Feller family, $L_0$ its $\rmC_0\times\rmC_0$-generator and $L$ its $\rmC_0\times\rmC$-generator. Then taking the closure in $\rmC_0(S)\times\rmC(S)$
\[
L_0=L\cap\rmC_0(S)\times\rmC_0(S),\quad\text{and}\quad L=\overline{L_0}.
\]
\end{proposition}
\begin{proof}
Firstly, we have $L_0\subset L\cap\rmC_0(S)\times\rmC_0(S)$ by Proposition \ref{propUniMgPb}. Hence $L\cap\rmC_0(S)\times\rmC_0(S) $ is an extension of $L_0$ satisfying the positive maximum principle, so by a maximality result (a consequence 
of Hille-Yoshida's, see for instance Lemma  19.12, p. 377 in \cite{Ka02}), $L_0= L\cap\rmC_0(S)\times\rmC_0(S)$.

Secondly, take $(f,g)\in L$. Let $h\in\rmC(S,\R_+^*)$ be a bounded function such that $hg\in\rmC_0(S)$. Thanks to Proposition \ref{propTCFel} the $\rmC_0\times\rmC_0$-generator of $(h\mycdot\Pbf_a)_a$ is $\overline{hL_0}
^{\rmC_0(S)\times\rmC_0(S)}$. Moreover the $\rmC_0\times\rmC$-generator of $(h\mycdot\Pbf_a)_a$ is $hL$. Hence applying the first step to the family $(h\mycdot\Pbf_a)_a$ we deduce that
\[
\overline{hL_0}^{\rmC_0(S)\times\rmC_0(S)}=(hL)\cap\rmC_0(S)\times\rmC_0(S),
\]
so $(f,hg)\in\overline{hL_0}^{\rmC_0(S)\times\rmC_0(S)}$ and $(f,g)\in\overline{L_0}^{\rmC_0(S)\times\rmC(S)}$.
\end{proof}

\begin{theorem}[Convergence of locally Feller family]\label{thmCvgLocFel}
For $n\in\N\cup\{\infty\}$, let $(\Pbf^n_a)_a\in\Pcal(\Dloc(S))^S$ be a locally Feller family and let $L_n$ be a subset of $\rmC_0(S)\times\rmC(S)$. Suppose that for any $n\in\N$, $\overline{L_n}$ is the generator of $(\Pbf^n_a)_a$, suppose also that $\rmD(L_\infty)$ is dense in $\rmC_0(S)$ and
\[
\Mc(L_\infty)=\{\Pbf^\infty_\mu\}_{\mu\in \Pcal(S^\Delta)}.
\]
Then we have equivalence between:
\begin{enumerate}
\item\label{it1thmCvgLocFel} the mapping
\[\begin{array}{ccc}
\N\cup\{\infty\}\times\Pcal(S^\Delta)&\to&\Pcal\left(\Dloc(S)\right)\\
(n,\mu)&\mapsto&\Pbf^n_\mu
\end{array}\]
is weakly continuous for the local Skorokhod topology;
\item\label{it2thmCvgLocFel} for any $a_n,a\in S$ such that $a_n\to a$, 
$\Pbf^n_{a_n}$ converges weakly for the local Skorokhod topology to $\Pbf^\infty_a$, 
as $n\to\infty$; 
\item\label{it3thmCvgLocFel} for any $(f,g)\in L_\infty$, there exist $(f_n,g_n)\in L_n$ such that $f_n\cv[\rmC_0]{n\to\infty} f$, $g_n\cv[\rmC]{n\to\infty}g$.
\end{enumerate}
\end{theorem}
\begin{remark}
1) We may deduce a similar theorem for Feller process.\\
2) An improvement with respect to the classical result of convergence Theorem 19.25, p. 385, in \cite{Ka02}, is that one does not need to know that $\overline{L_\infty}$ is the generator of the family, 
but only the fact that the martingale local problem is well-posed. Let us point out that there are situations were the generator is not known.
\end{remark}
\begin{proof}[Proof of Theorem \ref{thmCvgLocFel}]
It is straightforward that \ref{it1thmCvgLocFel}$\Rightarrow$\ref{it2thmCvgLocFel}. The implication \ref{it3thmCvgLocFel}$\Rightarrow$\ref{it1thmCvgLocFel} is a consequence of Proposition \ref{propUSCLMP}.\\
We prove that \ref{it2thmCvgLocFel}$\Rightarrow$\ref{it3thmCvgLocFel}. We can suppose that $L_\infty$ is the generator of $(\Pbf^\infty_a)_a$.
It is straightforward to obtain that
\[\begin{array}{ccc}
\N\cup\{\infty\}\times S^\Delta&\to&\Pcal\left(\Dloc(S)\right)\\
(n,a)&\mapsto&\Pbf^n_a
\end{array}\]
is weakly continuous for the local Skorokhod topology. 
Thanks to Proposition \ref{propEqLocGlo}, 
on the connection between $\Dloc(S)$ and $\D(S^\Delta)$, 
there exists $h\in\rmC(S,\R_+^*)$ such that, 
for any $n\in\N\cup\{\infty\}$ and $a\in S$, 
\[h\mycdot\Pbf^n_a\big(\Dloc(S)\cap\D(S^\Delta)\big)=1,\]
and the mapping
\[\begin{array}{ccc}
\N\cup\{\infty\}\times S^\Delta&\to&\Pcal\left(\D(S^\Delta)\right)\\
(n,a)&\mapsto&h\mycdot\Pbf^n_a
\end{array}\]
is weakly continuous for the global Skorokhod topology. Thanks to 
Theorem \ref{FirstcharFeller}, $(\Pbf^n_a)_a$ is a Feller family, for all $n\in\N\cup\{\infty\}$. From Remark \ref{rkTCGen} and Proposition \ref{propGenFF} we deduce that: $h\overline{L_n}\cap\rmC_0(S)\times \rmC_0(S)$ is the $\rmC_0\times\rmC_0$-generator of $(\Pbf^n_a)_a$ for $n\in\N$, $hL_\infty\cap\rmC_0(S)^2$ is the $\rmC_0\times\rmC_0$-generator of $(\Pbf^\infty_a)_a$ and
\[
\overline{hL_\infty\cap\rmC_0(S)\times\rmC_0(S)}^{\rmC_0(S)\times\rmC(S)}=hL_\infty.
\] 
Take arbitrary elements $a,a_1,a_2\ldots\in S^\Delta$ and $t,t_1,t_2\ldots\in \R_+$ such that $a_n\to a$ and $t_n\to t$, then $h\mycdot\Pbf^n_{a_n}$ converges weakly for the global Skorokhod topology to $h\mycdot\Pbf^\infty_{a}$.
By Theorem \ref{thmDefLFF}, $h\mycdot\Pbf^\infty_{a}$ is quasi-continuous, so $h\mycdot\Pbf^\infty_{a}(X_{t-}=X_t)=1$. Hence, for any $f\in\rmC_0(S)$
\[
h\mycdot \Ebf^n_{a_n}[f(X_{t_n})]\cv{n\to\infty}h\mycdot\Ebf^\infty_{a}[f(X_{t})].
\]
From here we can deduce that, for any $t\geq 0$
\[
\lim_{n\to\infty}\sup_{s\leq t}\sup_{a\in S}\big|h\mycdot \Ebf^n_{a}[f(X_{s})]-h\mycdot\Ebf^\infty_{a}[f(X_{s})]\big|=0.
\]
Here and elsewhere we denote by $\Ebf^n_a$ the expectation with respect to the 
probability measure $\Pbf^n_a$.
Hence by Trotter-Kato's theorem (cf. Theorem 19.25, p. 385, \cite{Ka02}), for any $(f,g)\in hL_\infty\cap\rmC_0(S)\times\rmC_0(S)$ there exist $(f_n,g_n)\in h\overline{L_n}\cap\rmC_0(S)\times\rmC_0(S)$ such that $(f_n,g_n)\cv{n\to\infty}(f,g)$, so it is straightforward to deduce statement \ref{it3thmCvgLocFel}.
\end{proof}

%


\subsection{Localisation for martingale problems and generators}

We are interested to the localisation procedure. More precisely, assume that 
$\mathcal{U}$ is a recovering of $S$ by open sets and, for each $U\in\mathcal{U}$, let $(\Pbf^U_a)_a$ be a locally Feller family, such that for all $U_1,U_2\in\mathcal{U}$ and $a\in S$
\[
\law_{\Pbf^{U_1}_a}\left(X^{\tau^{U_1\cap U_2}}\right) = \law_{\Pbf^{U_2}_a}\left(X^{\tau^{U_1\cap U_2}}\right).
\]
We wonder if there exists a locally Feller family $(\Pbf_a)_a$ such that for all $U\in\mathcal{U}$ and $a\in S$
\[
\law_{\Pbf_a}\big(X^{\tau^U}\big) = \law_{\Pbf^U_a}\big(X^{\tau^U}\big)\;\; ?
\]
An attempt to give a answer to this question needs to reformulate it in terms of generators 
of locally Feller families. This reformulation is suggested by the following:
\begin{proposition}\label{propEqGentoEqPro}
Let $L_1,L_2\subset \rmC_0(S)\times\rmC(S)$ be such that $\rmD(L_1)=\rmD(L_2)$ is dense in $\rmC_0(S)$ and take an open subset $U\subset S$. Suppose that
\begin{itemize}
\item[-] the martingale local problem associated to $L_1$ is well-posed, and,
 
\item[-]
for all $a\in U$ there exists $\Pbf^2\in\Mc(L_2)$ with $\Pbf^2(X_0=a)=1$. 
\end{itemize}
Then
\begin{align}\label{EqLmEqGentoEqPro}
\forall\Pbf^2\in\Mc(L_2),~\exists\Pbf^1\in\Mc(L_1),\quad\quad\law_{\Pbf^2}\left(X^{\tau^U}\right) = \law_{\Pbf^1}\left(X^{\tau^U}\right)
\end{align}
if and only if
\[
\forall (f,g)\in L_2,\quad\quad g_{|U}=(L_1f)_{|U}.
\]
\end{proposition}

We postpone the proof of this proposition and we state two main results of localisation.
\begin{theorem}[Localisation for the martingale problem]\label{thmLocMP}
Let $L$ be a linear subspace of $\rmC_0(S)\times\rmC(S)$ with $\rmD(L)$ dense in $\rmC_0(S)$. Suppose that for all $a\in S$ there exist a neighbourhood $V$ of $a$ and a subset $\widetilde{L}$ of $\rmC_0(S)\times\rmC(S)$ such that the martingale local problem associated to $\widetilde{L}$ is well-posed and such that
\begin{align}\label{eqPropLocMP}
\left\{(f,g_{|V})\mid(f,g)\in L\right\} = \left\{(f,g_{|V})\mid(f,g)\in \widetilde{L}\right\}.
\end{align}
Then the martingale local problem associated to $L$ is well-posed.
\end{theorem}
\begin{proof}
Thanks to Theorem \ref{thmExMP}, to prove the existence of a solution for the martingale local problem it suffices to prove that $L$ satisfies the positive maximum principle. Let $(f,g)\in L$ and $a\in S$ be such that $f(a)=\max f\geq 0$. Then there exist a neighbourhood $V$ of $a$ and a subset $\widetilde{L}$ of $\rmC_0(S)\times\rmC(S)$ such that the martingale local problem associated to $\widetilde{L}$ is well-posed and \eqref{eqPropLocMP}. In particular, by Theorem \ref{thmExMP}, $\widetilde{L}$ satisfies the positive maximum principle and so
\[
g(a)=\widetilde{L}f(a)\leq 0.
\]
To prove the uniqueness of the solution for the martingale local problem, we take $\Pbf^1,\Pbf^2\in\Mc(L)$ and an arbitrary open subset $V\Subset S$. By hypothesis and using the relative compactness of $V$, there exist $N\in\N$, open subsets $U_1,\ldots,U_N\subset S$ and subsets $L_1,\ldots,L_N\subset\rmC_0(S)\times\rmC(S)$ such that $V\Subset\bigcup_nU_n$, such that for all $1\leq n\leq N$ the martingale local problem associated to $L_n$ is well-posed and such that
\[
\left\{(f,g_{|U_n})\mid(f,g)\in L\right\} = \left\{(f,g_{|U_n})\mid(f,g)\in \widetilde{L_n}\right\}.
\]
At this level of the proof we need a technical but important result:
\begin{lemma}\label{lmMPtoSMP}
Let $U$ be an open subset of $S$ and $L$ be a subset of $\rmC_0(S)\times\rmC(S)$ such that $\rmD(L)$ is dense in $\rmC(S)$ and the martingale local problem associated to $L$ is well-posed. Then there exist a subset $L_0$ of $L$ and a function $h_0$ of $\rmC(S,\R_+)$ with $\{h_0\not =0\}=U$ such that $\overline{L}=\overline{L_0}$, such that $h_0L_0\subset\rmC_0(S)\times\rmC_0(S)$ and such that: for any $h\in\rmC(S,\R_+)$ with  $\{h\neq 0\}=U$ and  $\sup_{a\in U} (h/h_0)(a)<\infty$, the martingale problem associated to $(hL_0)^\Delta$ is well-posed in $\D(S^\Delta)$. Recall that $(hL_0)^\Delta$ is defined by \eqref{eqLDlt} and that the associated martingale problem is defined by \eqref{eqMgDlt}.
\end{lemma}

We postpone the proof of lemma to the Appendix (see \S \ref{secProofLmMPtoSMP}) and we proceed with the proof of our theorem.

Applying Lemma \ref{lmMPtoSMP}, there exist a subset $D$ of $\rmC_0(S)$ and a function $h$ of $\rmC(S,\R_+)$ with $\{h\not =0\}=V$ such that for all $1\leq n\leq N$: $\overline{L_n}=\overline{L_n\,_{|D}}$, $hL_n\,_{|D}\subset\rmC_0(S)\times\rmC_0(S)$ and the martingale problem associated to $\big(hL_n\,_{|D}\big)^\Delta$ is well-posed.
Denote $L_{N+1}:=D\times\{0\}$ and $U^{N+1}:=S^\Delta\backslash \overline{V}$. We may now apply Theorem 6.2 and also Theorem 6.1 pp. 216-217, in \cite{EK86}  to $hL_{|D}$ and $(U_n)_{1\leq n\leq N+1}$ and we deduce that the martingale problem associated to $(hL_{|D})^\Delta$ is well-posed. Hence $h\mycdot\Pbf^1=h\mycdot\Pbf^2$ so
\[
\law_{\Pbf^1}(X^{\tau^V})=\law_{\Pbf^2}(X^{\tau^V}).
\]
We obtain the result by letting $V$ to grow toward $S$. This ends the proof of the theorem except to the proof of Lemma \ref{lmMPtoSMP} postponed to \S \ref{secProofLmMPtoSMP}.
\end{proof}
\begin{theorem}[Localisation of generator]\label{thmLocGen}
Let $L$ be a linear subspace of $\rmC_0(S)\times\rmC(S)$ with $\rmD(L)$ dense in $\rmC_0(S)$. Suppose that for all subsets $V\Subset S$ there exists a linear subspace $\widetilde{L}$ of $\rmC_0(S)\times\rmC(S)$ such that $\overline{\widetilde{L}}$ is the generator of a locally Feller family and
\begin{align*}
\left\{(f,g_{|V})\mid(f,g)\in L\right\} = \left\{(f,g_{|V})\mid(f,g)\in \widetilde{L}\right\}.
\end{align*}
Then $\overline{L}$ is the generator of a locally Feller family.
\end{theorem}
\begin{proof} 
Thanks to Theorem \ref{thmLocMP} the martingale local problem associated to $L$ is well-posed, let $(\Pbf^\infty_a)_a$ the locally Feller family associate to $L$. Let $L_\infty$ be the generator of $(\Pbf^\infty_a)_a$.
Let $U_n\Subset S$ be an increasing sequence of open subsets such that $S=\bigcup_nU_n$ and  let $L_n\subset\rmC_0(S)\times\rmC(S)$ be such that for all $n\in\N$, $\overline{L_n}$ is the generator of a locally Feller family $(\Pbf^n_a)_a$ and
\begin{align}\label{eq1PropLocGen}
\left\{(f,g_{|U_n})\mid(f,g)\in L\right\} = \left\{(f,g_{|U_n})\mid(f,g)\in L_n\right\}. 
\end{align}
Then by using Proposition \ref{propEqGentoEqPro}, for all $n\in\N$ and $a\in S$
\begin{equation}\label{A18}
\law_{\Pbf^\infty_a}\left(X^{\tau^{U_n}}\right) =\law_{\Pbf^n_a}\left(X^{\tau^{U_n}}\right).
\end{equation}
At this level we use a result of localisation of the continuity stated and proved in \S \ref{secProofslocallyFeller}, Lemma \ref{lemLocCont}. Therefore, by \eqref{A18} the mapping
\[\begin{array}{ccc}
\N\cup\{\infty\}\times S^\Delta&\to&\Pcal(\Dloc(S))\\
(n,a)&\mapsto&\Pbf_a^n
\end{array}\]
is weakly continuous for the local Skorokhod topology.
Hence by Theorem \ref{thmCvgLocFel}, for any $f\in\rmD(L_\infty)$ there exists $(f_n)_n\in\rmD(L)^\N$ such that $(f_n,L_nf_n)\cv{n\to\infty}(f,L_\infty f)$, so by \eqref{eq1PropLocGen} $(f_n,Lf_n)\cv{n\to\infty}(f,L_\infty f)$. Hence $\overline{L}=L_\infty$ is the generator of a locally Feller family. The proof of the theorem is complete except for the proof of Proposition \ref{propEqGentoEqPro}.
\end{proof}
\begin{proof}[Proof of Proposition \ref{propEqGentoEqPro}]
Suppose \eqref{EqLmEqGentoEqPro}. For each $a\in U$, take an open subset $V\subset U$, $\Pbf^1\in\Mc(L_1)$ and $\Pbf^2\in\Mc(L_2)$ such that $a\in V\Subset S$ and $\Pbf^1(X_0=a)=\Pbf^2(X_0=a)=1$. By using the fourth part of Remark \ref{rkBLMP}  we have for each $(f,g)\in L_2$
\[
g(a)=\lim_{t\to 0}\frac{1}{t}\Big(\Ebf^2\left[f(X_{t\wedge\tau^V})\right]-f(a)\Big)=\lim_{t\to 0}\frac{1}{t}\Big(\Ebf^1\left[f(X_{t\wedge\tau^V})\right]-f(a)\Big)=L_1f(a).
\]
For the converse, by Lemma \ref{lmMPtoSMP} there exists $h\in\rmC(S,\R_+)$ with $\{h\not=0\}=U$ such that the martingale local problem associated to $hL_1=hL_2$ is well-posed. Take $\Pbf^2\in\Mc(L_2)$ and let $\Pbf^1\in\Mc(L_1)$ be such that $\law_{\Pbf^1}(X_0)=\law_{\Pbf^2}(X_0)$, then $h\mycdot\Pbf^1,h\mycdot\Pbf^2\in\Mc(hL_1)$ so $h\mycdot\Pbf^1=h\mycdot\Pbf^2$ and hence \eqref{EqLmEqGentoEqPro} is verified.
\end{proof}

\appendix\section{Appendix: proof of technical results} 
\subsection{Proofs of Propositions \ref{propLMPP} and \ref{propUSCLMP}}\label{preuves_entrelacees}
Remind that the proofs of Propositions \ref{propLMPP} and \ref{propUSCLMP}
are interlaced and will be performed in several ordered steps.
\begin{proof}[Proof of Lemma \ref{lmUQC}]
Take a metric $d$ on $S$ and $a_0\in \Kc$, then there exists $\eps_0>0$ such that $B(a_0,4\eps_0)\Subset S$ and $\left\{(a,b)\in S^2\mid a\in \Kc,~d(a,b)<3\eps_0\right\}\subset \Uc$. Define
\[
\widetilde{f}(a):=\left\{\begin{array}{ll}
1, & \text{if } d(a,a_0)\leq\eps_0,\\
0, & \text{if } d(a,a_0)\geq2\eps_0,\\
2-\frac{d(a,a_0)}{\eps_0},& \text{if } \eps_0\leq d(a,a_0)\leq2\eps_0.
\end{array}\right.\]
Then
\[\widetilde{f}\in\rmC_0(S),\quad 0\leq \widetilde{f}\leq 1,\quad\forall\,a\in B(a_0,\eps_0), \,\widetilde{f}(a) = 1\,\mbox{ and }\,\{ \widetilde{f}\not=0\}\subset B(a,3\eps_0).\]
Let $\eta>0$ be  arbitrary. There exist $(f,g)\in L$ and a sequence $(f_n,g_n)\in L_n$ 
such that $\|f-\widetilde{f}\|\leq\eta$ and the sequence $(f_n,g_n)_{n}$ converges to $(f,g)$ for the topology of $\rmC_0(S)\times\rmC(S)$.
Let $\tau_1\leq\tau_2$ be $(\Fc_{t+})_t$-stopping times and let $n$ be in $\N$, assume that $\Pbf\in\Mc(L_n)$. For $\eps<3\eps_0$ we denote
\[
\sigma_\eps:=\inf\Big\{t\geq \tau_1\,\Big|\, t\geq \xi\text{ or }\sup_{\tau_1\leq s\leq t}d(X_{\tau_1},X_s)\geq\eps \Big\}.
\]
Let an open subset $V\Subset S$ be such that $V\supset B(a_0,4\eps_0)$. If $t\geq 0$ and $\eps<3\eps_0$ we can write
\begin{align}\label{eq1LmUQC}\begin{split}
\Ebf\left[f_n(X_{t\wedge\tau^V\wedge\sigma_\eps\wedge\tau_2})\mathds{1}_{\{X_{\tau_1}\in B(a_0,\eps_0)\cap \Kc\}}\right]\hspace{-4cm}&\\
&=\Ebf\left[\Big(f_n(X_{t\wedge\tau^V\wedge\tau_1})+\int_{t\wedge\tau^V\wedge\tau_1}^{t\wedge\tau^V\wedge\sigma_\eps\wedge\tau_2}g_n(X_s)\d s\Big)\mathds{1}_{\{X_{\tau_1}\in B(a_0,\eps_0)\cap \Kc\}}\right]\\
&\geq\Ebf\left[\widetilde{f}(X_{t\wedge\tau^V\wedge\tau_1})\mathds{1}_{\{X_{\tau_1}\in B(a_0,\eps_0)\cap \Kc\}}\right]-\|\widetilde{f}-f_n\|\\
&\quad+\Ebf\left[\int_{t\wedge\tau^V\wedge\tau_1}^{t\wedge\tau^V\wedge\sigma_\eps\wedge\tau_2}g_n(X_s)\d s\mathds{1}_{\{X_{\tau_1}\in B(a_0,\eps_0)\cap \Kc\}}\right]\\
&\geq \Pbf\big(X_{\tau_1}\in B(a_0,\eps_0)\cap \Kc\big)-\Pbf\big(t\wedge\tau^V<\tau_1<\xi\big)-\eta-\|f-f_n\|\\
&\quad-\Ebf\big[(\tau_2-\tau_1)\mathds{1}_{\{X_{\tau_1}\in \Kc\}}\big]\cdot\|g_n\|_{B(a_0,4\eps_0)}.
\end{split}\end{align}
Splitting on the events $\{\sigma_\eps>\tau_2\}$, $\{\sigma_\eps\leq t\wedge\tau^V\wedge \tau_2\}$ and $\{t\wedge\tau^V<\sigma_\eps\leq\tau_2\}$
\begin{align}\label{eq2LmUQC}\begin{split}
\Ebf\left[f_n(X_{t\wedge\tau^V\wedge\sigma_\eps\wedge\tau_2})\mathds{1}_{\{X_{\tau_1}\in B(a_0,\eps_0)\cap \Kc\}}\right]\hspace{-5cm}&\\
&\leq \Pbf\big(X_{\tau_1}\in B(a_0,\eps_0)\cap \Kc,~\sigma_\eps>\tau_2\big)+\eta+\|f-f_n\|\\
&\quad+ \Ebf\big[f_n(X_{\sigma_\eps})\mathds{1}_{\{X_{\tau_1}\in B(a_0,\eps_0)\}}\big]
+\Pbf\big(X_{\tau_1}\in \Kc,~t<\tau_2\big)+\eta+\|f-f_n\|.
\end{split}\end{align}
Hence by \eqref{eq1LmUQC} and \eqref{eq2LmUQC},
\begin{multline*}
\Pbf\big(X_{\tau_1}\in B(a_0,\eps_0)\cap \Kc,~\tau(\tau_1)\leq\tau_2\big)
\leq\Pbf\big(X_{\tau_1}\in B(a_0,\eps_0)\cap \Kc,~\sigma_\eps\leq\tau_2\big)\\
\leq 3\eta+3\|f-f_n\| 
+\Pbf\big(t\wedge\tau^V<\tau_1<\xi\big)
+\Ebf\big[(\tau_2-\tau_1)\mathds{1}_{\{X_{\tau_1}\in \Kc\}}\big]\cdot \|g_n\|_{B(a_0,4\eps_0)}\\
+ \Ebf\big[f_n(X_{\sigma_\eps})\mathds{1}_{\{X_{\tau_1}\in B(a_0,\eps_0)\}}\big]
+\Pbf\big(X_{\tau_1}\in \Kc,~t<\tau_2\big).
\end{multline*}
Since the limit $\lim_{\eps\uparrow3\eps_0}X_{\sigma_\eps}$ exists and is in $S^\Delta\backslash B(X_{\tau_1},3\eps_0)$ we have
\begin{align*}
\limsup_{\eps\uparrow3\eps_0}\Ebf\big[f_n(X_{\sigma_\eps})\mathds{1}_{\{X_{\tau_1}\in B(a_0,\eps_0)\}}\big]
&\leq \|f_n\|_{B(a_0,2\eps_0)^c}
\leq \|f-f_n\|+\|f-\widetilde{f}\|+\|\widetilde{f}\|_{B(a_0,2\eps_0)^c}\\
&\leq \|f-f_n\|+\delta,
\end{align*}
so
\begin{align*}
\Pbf\big(X_{\tau_1}\in B(a_0,\eps_0)\cap \Kc,~\tau(\tau_1)\leq\tau_2\big)\hspace{-3cm}&\\
&\leq 4\eta+4\|f-f_n\|
+\Pbf\big(t\wedge\tau^V<\tau_1<\xi\big)\\
&\quad+\Ebf\big[(\tau_2-\tau_1)\mathds{1}_{\{X_{\tau_1}\in \Kc\}}\big]\cdot\|g_n\|_{B(a_0,4\eps_0)}
+\Pbf\big(X_{\tau_1}\in \Kc,~t<\tau_2\big).
\end{align*}
Letting $t\to\infty$ and $V$ growing to $S$, $\Pbf\big(t\wedge\tau^V<\tau_1<\xi\big)$ tends to $0$, hence
\begin{multline*}
\Pbf\big(X_{\tau_1}\in B(a_0,\eps_0)\cap \Kc,~\tau(\tau_1)\leq\tau_2\big)\\
\leq 4\eta+4\|f-f_n\|
+\Ebf[(\tau_2-\tau_1)\mathds{1}_{\{X_{\tau_1}\in \Kc\}}]\cdot\|g_n\|_{B(a_0,4\eps_0)}\\
+\Pbf\big(X_{\tau_1}\in \Kc,~\tau_2=\infty\big).
\end{multline*}
So letting $n\to\infty$, $\Ebf[(\tau_2-\tau_1)\mathds{1}_{\{X_{\tau_1}\in \Kc\}}]\to 0$ and $\eta\to 0$ we deduce that for each $\eps>0$ there exist $n_0\in\N$ and $\delta>0$ such that: for any $n\geq n_0$, $(\Fc_{t+})_t$-stopping times $\tau_1\leq\tau_2$ and $\Pbf\in\Mc(L_n)$ satisfying $\Ebf[(\tau_2-\tau_1)\mathds{1}_{\{X_{\tau_1}\in \Kc\}}]\leq\delta$ we have
\[
\Pbf(X_{\tau_1}\in B(a_0,\eps_0)\cap \Kc,~\tau(\tau_1)\leq \tau_2)\leq\eps.
\]
We conclude since $a_0$ was arbitrary chosen in $\Kc$ and by using a finite recovering of the compact $\Kc$.
\end{proof}
\begin{proof}[Proof of part \ref{it4PropLMPP} of Proposition \ref{propLMPP}]~

\emph{Step 1: we prove the $(\Fc_{t+})_t$-quasi-continuity before the explosion time $\xi$.}
Let $\tau_n,\tau$ be $(\Fc_{t+})_t$-stopping times and  denote $\widetilde{\tau}_n:=\inf_{m\geq n}\tau_m$, $\widetilde{\tau}:=\sup_{n\in\N}\widetilde{\tau}_n$ and
\[A:=\left\{\begin{array}{ll}
\lim_{n\to\infty}X_{\widetilde{\tau}_n},&\text{if the limit exists,}\\
\Delta,&\text{otherwise.}
\end{array}\right.\]
Let $d$ be a metric on $S^\Delta$ and take $\eps>0$, $t\geq 0$ and an open subset $U\Subset S$.  Since 
\[\lim_{n\to\infty}\Ebf\big[\widetilde{\tau}\wedge t\wedge\tau^U-\widetilde{\tau}_n\wedge t\wedge\tau^U\big]=0,\]
by Lemma \ref{lmUQC} applied to $\Kc:=\overline{U}$ and $\Uc=\left\{(a,b)\in S^2\mid d(a,b)<\eps\right\}$ we get
\[
\Pbf\big(X_{\widetilde{\tau}_n\wedge t\wedge\tau^U}\in U,~d(X_{\widetilde{\tau}_n\wedge t\wedge\tau^U},X_{\widetilde{\tau}\wedge t\wedge\tau^U})\geq\eps\big)\cv{n\to\infty}0.
\]
Hence
\begin{align*}
\Pbf\big(\widetilde{\tau}\leq t\wedge\tau^U,~d(X_{\widetilde{\tau}_n},X_{\widetilde{\tau}})\geq\eps\big)
&=\Pbf\big(\widetilde{\tau}_n<\widetilde{\tau}\leq t\wedge\tau^U,~d(X_{\widetilde{\tau}_n},X_{\widetilde{\tau}})\geq\eps\big)\\
&\leq\Pbf\big(X_{\widetilde{\tau}_n\wedge t\wedge\tau^U}\in U,~d(X_{\widetilde{\tau}_n\wedge t\wedge\tau^U},X_{\widetilde{\tau}\wedge t\wedge\tau^U})\geq\eps\big).
\end{align*}
Letting $n\to\infty$ on the both sides of the latter inequality  we obtain that
\[
\Pbf\big(\widetilde{\tau}\leq t\wedge\tau^U,~d(A,X_{\widetilde{\tau}})\geq\eps\big)=0.
\]
Then, successively if $t\to\infty$, $U$ growing to $S$ and $\eps\to 0$ it follows that
\[
\Pbf\big(\widetilde{\tau}<\infty,~\{X_s\}_{s<\widetilde{\tau}}\Subset S,~A\not=X_{\widetilde{\tau}}\big)=0.
\]
We deduce
\begin{align}
\nonumber&\Pbf\big(X_{\tau_n}\underset{n\to\infty}{{\longrightarrow\hspace{-0.45cm}\not\hspace{0.45cm}}}X_\tau,~\tau_n\cv{n\to\infty}\tau<\infty,~\{X_s\}_{s<\tau}\Subset S\big)\\
\label{eqPropQC}&\hspace{3cm}=\Pbf\big(A\not=X_{\widetilde{\tau}},~\tau_n\cv{n\to\infty}\tau=\widetilde{\tau}<\infty,~\{X_s\}_{s<\widetilde{\tau}}\Subset S\big)=0.
\end{align}
\emph{Step 2: we prove that $\Pbf\big(\Dloc(S)\cap\D(S^\Delta)\big)=1$.}
Let $K$ be a compact subset of $S$ and take an open subset $U\Subset S$ containing $K$. For $n\in\N$ define the stopping times
\begin{align*}
&\sigma_0:=0,\\
&\tau_n:=\inf\left\{t\geq\sigma_n\mid\{X_s\}_{\sigma_n\leq s\leq t}\not\Subset S\backslash K\right\},\\
&\sigma_{n+1}:=\inf\left\{t\geq\tau_n\mid\{X_s\}_{\tau_n\leq s\leq t}\not\Subset U\right\}.
\end{align*}
Let $V_k\Subset S\backslash K$ be an increasing sequence of open subset such that $S\backslash K=\bigcup_kV_k$, and denote $\tau_n^k:=\inf\left\{t\geq\sigma_n\mid\{X_s\}_{\sigma_n\leq s\leq t}\not\Subset V_k\right\}$. Then, by \eqref{eqPropQC}
\[
\Pbf\Big(X_{\tau_n^k}\underset{k\to\infty}{{\longrightarrow\hspace{-0.45cm}\not\hspace{0.45cm}}}X_{\tau_n},~\tau_n<\infty,~\{X_s\}_{s<\tau_n}\Subset S\Big)=0,
\]
so $\{\tau_n<\xi\}=\{X_{\tau_n}\in K\}$ $\Pbf$-almost surely.
Thanks to Lemma \ref{lmUQC} applied to $\Kc:=K$ and $\Uc:=U^2\cup(S\backslash K)^2$
\[
\sup_{n\in\N}\Pbf\big(X_{\tau_n}\in K,~\sigma_{n+1}<\tau_n+\eps\big)\cv{\eps\to0}0.
\]
For $\eps>0$,
\begin{multline*}
\Pbf\big(\xi<\infty,~\{X_s\}_{s<\xi}\not\Subset S\text{ and }\forall t<\xi,\exists s\in[t,\xi),~X_s\in K\big)\\
\leq \Pbf\big(\exists n,\forall m\geq n,~\tau_m<\xi<\tau_m+\eps\big)
\leq \sup_{n\in\N}\Pbf\big(\tau_n<\xi<\tau_n+\eps\big)\\
\leq \sup_{n\in\N}\Pbf\big([X_{\tau_n}\in K,~\sigma_{n+1}<\tau_n+\eps\big),
\end{multline*}
so letting $\eps\to0$ we obtain 
\begin{equation}\label{prepaKcompact}
\Pbf\big(\xi<\infty,~\{X_s\}_{s<\xi}\not\Subset S\text{ and }\forall t<\xi,\exists s\in[t,\xi),~X_s\in K\big)=0.
\end{equation}
Letting $K$ growing toward $S$, we deduce from \eqref{prepaKcompact} that $\Pbf\big(\Dloc(S)\cap\D(S^\Delta)\big)=1$.\\
\emph{Step 3.}
Let $\tau_n,\tau$ be $(\Fc_{t+})$-stopping times. By the first step $X_{\tau_n}\cv{n\to\infty} X_\tau$ $\Pbf$-almost surely on \[\big\{\tau_n\cv{n\to\infty}\tau<\infty,~\{X_s\}_{s<\tau}\Subset S\big\},\] 
by the second step this is also the case on \[\big\{\tau_n\cv{n\to\infty}\tau=\xi<\infty,~\{X_s\}_{s<\tau}\not\Subset S\big\},\] and this is clearly true on $\big\{\tau_n\cv{n\to\infty}\tau>\xi\big\}$, so the proof 
is done.
\end{proof}
\begin{proof}[Proof of part \ref{it1PropLMPP} of Proposition \ref{propLMPP}]
Take $(f,g)\in L$ and an open subset $U\Subset S$. If  $s_1\leq\cdots \leq s_k\leq s\leq t$ are positive numbers  and  $f_1,\ldots ,f_k\in\rmC(S^\Delta)$, we need to prove that
\begin{align}\label{eqPropTCMP}
h\mycdot\Ebf\left[\left(f(X_{t\wedge\tau^U})-f(X_{s\wedge\tau^U})-\int_{s\wedge\tau^U}^{t\wedge\tau^U}(hg)(X_u)\d u\right)f_1(X_{s_1})\cdots f_k(X_{s_k})\right]=0.
\end{align}
We will proceed in two steps: firstly we suppose that $U\Subset\{h\not = 0\}$. Recalling the definition \eqref{taug}, if we denote  $\tau_t:=\tau^h_t\wedge\tau^U$, then we have, for all $t\in\R_+$, 
\begin{align}\label{tauh1}
h\mycdot X_{t\wedge\tau^U(h\mycdot X)}=X_{\tau_t},
\end{align}
\begin{align}\label{tauh2}
\int_0^{t\wedge\tau^U(h\mycdot X)}(hg)(h\mycdot X_u)\d u=\int_0^{t\wedge\tau^U(h\mycdot X)}(hg)(X_{\tau_u})\d u=\int_0^{\tau_t}g(X_u)\d u.
\end{align}
Hence by \eqref{tauh1}-\eqref{tauh2} and optional sampling Theorem \ref{thmOptSam}
\begin{multline*}
h\mycdot\Ebf\left[\left(f(X_{t\wedge\tau^U})-f(X_{s\wedge\tau^U})-\int_{s\wedge\tau^U}^{t\wedge\tau^U}(hg)(X_u)
\d u\right)f_1(X_{s_1})\cdots f_k(X_{s_k})\right]\\
=h\mycdot\Ebf\left[\left(f(X_{t\wedge\tau^U})-f(X_{s\wedge\tau^U})-\int_{s\wedge\tau^U}^{t\wedge\tau^U}(hg)(X_u)
\d u\right)f_1(X_{s_1\wedge\tau^U})\cdots f_k(X_{s_k\wedge\tau^U})\right]\\
=\Ebf\left[\left(f(X_{\tau_t})-f(X_{\tau_s})-\int_{\tau_s}^{\tau_t}g(X_u)\d u\right)f_1(X_{\tau_{s_1}})\cdots f_k(X_{\tau_{s_k}})\right]=0.
\end{multline*}
Secondly, we suppose only that $U\Subset S$. Let $d$ be a metric on $S$ and we introduce, for $n\geq 1$ integer, $U_n:=\big\{a\in U\,|\, d(a,\{h=0\})>n^{-1}\big\}$. Then it is straightforward to obtain the pointwise convergences
\begin{eqnarray*}
&h\mycdot X_{t\wedge\tau^{U_n}(h\mycdot X)}\cv{n\to\infty}h\mycdot X_{t\wedge\tau^U(h\mycdot X)},\\
&\int_0^{t\wedge\tau^{U_n}(h\mycdot X)}(hg)(h\mycdot X_u)\d u\cv{n\to\infty}\int_0^{t\wedge\tau^U(h\mycdot X)}(hg)(h\mycdot X_u)\d u,
\end{eqnarray*}
so
\begin{align*}
&f(X_{t\wedge\tau^{U_n}})-f(X_{s\wedge\tau^{U_n}})-\int_{s\wedge\tau^{U_n}}^{t\wedge\tau^{U_n}}(hg)(X_u)\d u\\
&\hspace{3cm}\cv[h\mycdot \Pbf \text{-a.s.}]{n\to\infty}~f(X_{t\wedge\tau^U})-f(X_{s\wedge\tau^U})-\int_{s\wedge\tau^U}^{t\wedge\tau^U}(hg)(X_u)\d u.
\end{align*}
Applying the first step to $U_n$ and letting $n\to\infty$, by dominated convergence we obtain \eqref{eqPropTCMP}.
\end{proof}
\begin{proof}[Proof of part \ref{it1PropUSCLMP} of Proposition \ref{propUSCLMP}]
By using Proposition \ref{propEqLocGlo} we know that there exists $h\in\rmC(S,\R_+^*)$ such that $\Dloc(S)\cap\D(S^\Delta)$ has probability $1$ under $h\mycdot\Pbf^n$ and under $h\mycdot\Pbf$ and such that $h\mycdot\Pbf^n$ converges weakly to $h\mycdot\Pbf$ for the global Skorokhod topology from $\D(S^\Delta)$. Let us fix  $(f,g)$ 
and  $(f_n,g_n)$ arbitrary as in \eqref{cvgenerateur} and then we can modify $h$ 
such that it satisfies furthermore $hg_n,hg\in\rmC_0(S)$ and $hg_n\cv[\rmC_0]{n\to\infty}hg$.

Let $\Tbb$ be the set of $t\in \R_+$ such that $h\mycdot\Pbf(X_{t-}=X_t)=1$, so $\R_+\backslash\Tbb$ is countable. Let $s_1\leq \cdots\leq s_k\leq s\leq t$ belonging to $\Tbb$ and let $\varphi_1,\ldots,\varphi_k\in\rmC(S^\Delta)$ be. 
By using \ref{it1PropLMPP} of Proposition \ref{propLMPP} and the first part of Remark \ref{rkBLMP}
\begin{align}\label{eqLmCP}
h\mycdot \Ebf^n\left[\left(f_n(X_t)-f_n(X_s)-\int_s^t(hg_n)(X_u)\d u\right)\varphi_1(X_{s_1})\cdots \varphi_k(X_{s_k})\right]=0.
\end{align}
The sequence of functions $\big(f_n(X_t)-f_n(X_s)-\int_s^t(hg_n)(X_u)\d u\big)\varphi_1(X_{s_1})\cdots \varphi_k(X_{s_k})$ converges uniformly to the function $\big(f(X_t)-f(X_s)-\int_s^t(hg)(X_u)\d u\big)\varphi_1(X_{s_1})\cdots \varphi_k(X_{s_k})$ which is continuous $h\mycdot\Pbf$-almost everywhere for the topology of $\D(S^\Delta)$. Hence we can take the limit, as $n\to\infty$, in \eqref{eqLmCP} and we obtain that
\begin{align}\label{eqLmCP2}
h\mycdot\Ebf\left[\left(f(X_t)-f(X_s)-\int_s^t(hg)(X_u)\d u\right)\varphi_1(X_{s_1})\cdots \varphi_k(X_{s_k})\right]=0.
\end{align}
Since $\Tbb$ is dense in $\R_+$, by right continuity of paths of the canonical process, and 
by dominated convergence \eqref{eqLmCP2} extends to $s_i,s,t\in\R_+$. Hence $h\mycdot\Pbf\in\Mc(\{(f,hg)\})$, so using \eqref{eqProp2TC} and part \ref{it1PropLMPP} of Proposition \ref{propLMPP}, $\Pbf=(1/h)\mycdot h\mycdot\Pbf\in\Mc(\{(f,g)\})$. Since $(f,g)\in L$ was chosen arbitrary, we have proved that $\Pbf\in\Mc(L)$.
\end{proof}
\begin{proof}[Proof of part \ref{it2PropLMPP} of Proposition \ref{propLMPP}]
It is straightforward that $\Mc(\vect(L))=\Mc(L)$. Let $\Pbf\in\Mc(L)$. We apply 
the part \ref{it1PropUSCLMP} of Proposition \ref{propUSCLMP} to the stationary 
sequences $\Pbf^n=\Pbf$ and  $L_n=\vect(L)$ and to $\overline{\vect(L)}$. 
Hence $\Pbf\in\Mc(\overline{\vect(L)})$ and the proof is finished.
\end{proof}
%
\begin{proof}[Proof of part \ref{it2PropUSCLMP} of Proposition \ref{propUSCLMP}]
Take $t\in\R_+$ and an open subset $U\Subset S$, and let $d$ be a metric on $S^\Delta$. By Lemma \ref{lmUQC}, considering $\Kc:=\overline{U}$ and $\Uc:=\left\{(a,b)\in S^2\mid d(a,b)<\eps\right\}$, we have
\[
\sup_{\substack{\tau_1\leq\tau_2\\\tau_2\leq(\tau_1+\delta)\wedge\tau^U\wedge t}}\Pbf_n\big(d(X_{\tau_1},X_{\tau_2})\geq \eps\big)
\cv{\substack{n\to\infty\\\delta\to 0}}0,
\]
hence \eqref{hypropAldous} is satisfied and we can apply the Aldous criterion (see also Proposition 2.14 in \cite{GH017}).
\end{proof}
\begin{proof}[Proof of part \ref{it3PropLMPP} of Proposition \ref{propLMPP}]
It is straightforward that $\Mc(L)$ is convex. To prove the compacteness, let $(\Pbf^n)_n$ be a sequence from $\Mc(L)$. We apply the part \ref{it2PropUSCLMP} of Proposition \ref{propUSCLMP} 
to this sequence and  to the stationary sequence  $L_n=L$. Hence $(\Pbf^n)$ is tight, so there exists a subsequence $(\Pbf^{n_k})_k$ which converges toward some $\Pbf\in\Pcal(\Dloc(S))$. Thanks 
to the part \ref{it1PropUSCLMP} of Proposition \ref{propUSCLMP} we can deduce that 
$\Pbf\in\Mc(L)$. The statement of the proposition is then obtained since $\Pcal(\Dloc(S))$ is a Polish space.
\end{proof}

\subsection{Proof of Theorem \ref{thmDefLFF}}\label{secProofslocallyFeller}
To prove the theorem we will use three preliminary results. 
\begin{lemma}\label{lmMpCimSM}
Let $(\Pbf_a)_a\in\Pcal(\Dloc(S))^S$ be such that $a\mapsto\Pbf_a$ is continuous for the local Skorokhod topology. Suppose that for all $a\in S^\Delta$: $\Pbf_a(X_0=a)=1$ and there exists a dense subset $\Tbb_a\subset\R_+$ such that for any $B\in\Fc$ and $t_0\in\Tbb_a$
\[
\Pbf_a\left((X_{t_0+t})_t\in B\mid\Fc_{t_0}\right)=\Pbf_{X_{t_0}}(B)\quad\Pbf_a\text{-almost surely}.
\]
Then $(\Pbf_a)_a$ is a $(\Fc_{t+})_t$-strong Markov family.
\end{lemma}
\begin{proof}
Let $\tau$ be a $(\Fc_{t+})_t$-stopping time, let $a\in S$ be and let $F$ be a bounded continuous function from $\Dloc(S)$ to $\R$. 
For each $n\in\N^*$ chose a discrete subspace $\Tbb_a^n\subset\Tbb_a$ such that $(t,t+n^{-1}]\cap\Tbb_a^n$ is not empty for any $t\in\R_+^*$, and define
\[
\tau_n:=\min\left\{t\in\Tbb_a^n\mid \tau<t\right\}.
\]
Hence $\tau_n$ is a $(\Fc_t)_t$-stopping time with value in $\Tbb_a^n$ , so
\[
\Ebf_a\left[F\left((X_{\tau_n+t})_t\right)\mid\Fc_{\tau_n}\right]=\Ebf_{X_{\tau_n}}F\quad\Pbf_a\text{-almost surely.}
\]
Since $\tau<\tau_n\leq\tau+n^{-1}$ on $\{\tau<\infty\}$ and $a\mapsto\Pbf_a$ is continuous, $\lim_{n\to\infty}\Ebf_{X_{\tau_n}}F=\Ebf_{X_\tau}F$. We have
\begin{align}
\nonumber&\Ebf_a\left|\Ebf_a\left[F\left((X_{\tau+t})_t\right)\mid\Fc_{\tau+}\right]-\Ebf_a\left[F\left((X_{\tau_n+t})_t\right)\mid\Fc_{\tau_n}\right]\right|\\
\label{eqPrLmMpCimSM}&\hspace{3cm}\leq\Ebf_a\left|\Ebf_a\left[F\left((X_{\tau+t})_t\right)\mid\Fc_{\tau+}\right]-\Ebf_a\left[F\left((X_{\tau+t})_t\right)\mid\Fc_{\tau_n}\right]\right|\\
\nonumber&\hspace{3cm}\quad+\Ebf_a\left|F\left((X_{\tau+t})_t\right)-F\left((X_{\tau_n+t})_t\right)\right|.
\end{align}
On the right hand side, the first term converges to $0$ (see, for instance, Theorem 7.23, p. 132 in \cite{Ka02}) and the second term converges to $0$ by dominated convergence. Hence
\[
\Ebf_a\left[F\left((X_{\tau+t})_t\right)\mid\Fc_{\tau+}\right]=\Ebf_{X_{\tau}}F\quad\Pbf_a\text{-almost surely,}
\]
so $(\Pbf_a)_a$ is a $(\Fc_{t+})_t$-strong Markov family.
\end{proof}
\begin{lemma}[Localisation of continuity]\label{lemLocCont}
Set $\widetilde{S}$ an arbitrary metrisable topological space, consider $U_n\subset S$, an increasing sequence of open subsets such that $S=\bigcup_nU_n$. Let $(\Pbf_a^n)_{a,n}\in\Pcal(\Dloc(S))^{\widetilde{S}\times \N}$ be such that 
\begin{enumerate}
\item\label{itLemLocCont1} for each $n\in\N$, $a\mapsto\Pbf_a^n$ is weakly continuous for the local Skorokhod topology, 
\item\label{itLemLocCont2} for each $n\leq m$ and $a\in \widetilde{S}$
\begin{align}\label{eq1LocCont}
\law_{\Pbf^m_a}\left(X^{\tau^{U_n}}\right) =\law_{\Pbf^n_a}\left(X^{\tau^{U_n}}\right).
\end{align}
\end{enumerate}
Then there exists a unique family $(\Pbf_a^\infty)_{a}\in\Pcal(\Dloc(S))^{\widetilde{S}}$ such that for any $n\in\N$ and $a\in \widetilde{S}$
\begin{align}\label{eq2LocCont}
\law_{\Pbf^\infty_a}\left(X^{\tau^{U_n}}\right) =\law_{\Pbf^n_a}\left(X^{\tau^{U_n}}\right).
\end{align}
Furthermore the mapping
\begin{align}\label{eq3LocCont}\begin{array}{ccc}
\N\cup\{\infty\}\times \widetilde{S}&\to&\Pcal(\Dloc(S))\\
(n,a)&\mapsto&\Pbf_a^n
\end{array}\end{align}
is weakly continuous for the local Skorokhod topology.
\end{lemma}

Before giving the proof of this lemma let us recall that in  Theorem 2.15 of \cite{GH017} 
is obtained an improvement of the Aldous criterion of tightness. 
More precisely a subset $\Pcal\subset\Pcal\left(\Dloc(S)\right)$ is tight if and 
only if 
\begin{equation}\label{part3thmTension}
\forall t\geq 0,~\forall\eps>0,~\forall\text{ open } U\Subset S,\quad\sup_{\Pbf\in\Pcal}\sup_{\substack{\tau_1\leq\tau_2\leq\tau_3\\ \tau_3\leq (\tau_1+\delta)\wedge t\wedge\tau^U}}\Pbf(R\geq\eps)\cv{\delta\to 0}0,
\end{equation}
where the supremum is taken along  $\tau_i$ stopping times and with
\[
R:=\left\{\begin{array}{ll}
d(X_{\tau_1},X_{\tau_2})\wedge d(X_{\tau_2},X_{\tau_3})&\text{if }0<\tau_1<\tau_2,\\
d(X_{\tau_2-},X_{\tau_2})\wedge d(X_{\tau_2},X_{\tau_3})&\text{if }0<\tau_1=\tau_2,\\
d(X_{\tau_1},X_{\tau_2})&\text{if }0=\tau_1,
\end{array}\right.
\]
$d$ being an arbitrary  metric on $S^\Delta$.
\begin{proof}[Proof of Lemma \ref{lemLocCont}]
The uniqueness is straightforward using that $X^{\tau^{U_n}}$ converge to $X$ pointwise for the local Skorokhod topology as $n\to\infty$.

Let us prove that for any compact subset $K\subset \widetilde{S}$, the set $\left\{\Pbf_a^n\mid a\in K,~n\in\N\right\}$ is tight. If $U\Subset S$ is 
an arbitrary open subset,  there exists $N\in\N$ such that $U\subset U_N$.
Let $t.\eps>0$ be. By the continuity of $a\mapsto\Pbf_a^n$, the set $\left\{\Pbf_a^n\mid a\in K,~0\leq n\leq N\right\}$ is tight, so using the characterisation \eqref{part3thmTension} we have
\[
\sup_{\substack{0\leq n\leq N\\ a\in K}}\sup_{\substack{\tau_1\leq\tau_2\leq\tau_3\\ \tau_3\leq (\tau_1+\delta)\wedge t\wedge\tau^U}}\Pbf^n_a(R\geq\eps)\cv{\delta\to 0}0.
\]
Since $U\subset U_N$, for all $n\geq N$ and $a\in K$,
\[
\law_{\Pbf^N_a}\left(X^{\tau^{U}}\right) =\law_{\Pbf^n_a}\left(X^{\tau^{U}}\right),
\]
hence
\[
\sup_{n\in\N,~a\in K}\sup_{\substack{\tau_1\leq\tau_2\leq\tau_3\\ \tau_3\leq (\tau_1+\delta)\wedge t\wedge\tau^U}}\Pbf^n_a(R\geq\eps)
=\sup_{\substack{0\leq n\leq N\\ a\in K}}\sup_{\substack{\tau_1\leq\tau_2\leq\tau_3\\ \tau_3\leq (\tau_1+\delta)\wedge t\wedge\tau^U}}\Pbf^n_a(R\geq\eps)
\cv{\delta\to 0}0.
\]
So, again by \eqref{part3thmTension}, $\left\{\Pbf_a^n\mid a\in K,~n\in\N\right\}$ is tight.

Hence, if  $a\in\widetilde{S}$, then the set $\{\Pbf_a^n\}_{n}$ is tight. Fix such $a$, there exist an increasing sequence $\varphi(k)$ and a probability measure $\Pbf_a^\infty\in\Pcal(\Dloc(S))$ such that $\Pbf_a^{\varphi(k)}$ converges to $\Pbf_a^\infty$ as $k\to\infty$. Fix an arbitrary $n\in\N$, there exists $k_0\in\N$ such that $\varphi(k_0)\geq n$ and $U_n\Subset U_{\varphi(k_0)}$. Thanks to Proposition \ref{propEqLocGlo}, there exists $g\in\rmC(S,\R_+)$ such that $U_{\varphi(k_0)}=\{g\not =0\}$ and
such that $g\mycdot\Pbf_a^{n_k}$ converges to $g\mycdot\Pbf_a^\infty$ weakly for the local Skorokhod topology,
as $k\to\infty$.  By using \eqref{eq1LocCont} we have, for each $k\geq k_0$, $g\mycdot\Pbf_a^{\varphi(k)}=g\mycdot\Pbf_a^{\varphi(k_0)}$, so $g\mycdot\Pbf_a^\infty=g\mycdot\Pbf_a^{\varphi(k_0)}$. Hence we deduce
\[
\law_{\Pbf^\infty_a}\left(X^{\tau^{U_n}}\right) =\law_{\Pbf^{\varphi(k_0)}_a}\left(X^{\tau^{U_n}}\right)=\law_{\Pbf^n_a}\left(X^{\tau^{U_n}}\right).
\]

Let us prove that the mapping in \eqref{eq3LocCont} is weakly continuous for the local Skorokhod topology.
Since we already verified the tightness it suffices to prove that: for any sequences $n_k\in\N\cup\{\infty\}$, $a_k\in \widetilde{S}$ such that $n_k\to\infty$ and $a_k\to a\in \widetilde{S}$ as $k\to\infty$ and such that the sequence $\Pbf^{n_k}_{a_k}$ converges to $\Pbf\in\Pcal(\Dloc(S))$, then $\Pbf=\Pbf_a^\infty$.
Fix an arbitrary $N\in\N$, there exists $k_0\in\N$ such that $n_{k_0}\geq N$ and $U_N\Subset U_{n_{k_0}}$. As previously, by using Proposition \ref{propEqLocGlo} again, there exists $g\in\rmC(S,\R_+)$ such that $U_{n_{k_0}}=\{g\not =0\}$,  $g\mycdot\Pbf_{a_k}^{n_k}$ converges to $g\mycdot\Pbf$ and $g\mycdot\Pbf_{a_k}^{n_{k_0}}$ converges to $g\mycdot\Pbf_a^{n_{k_0}}$, as $k\to\infty$. Thanks to \eqref{eq2LocCont} $g\mycdot\Pbf_{a_k}^{n_k}=g\mycdot\Pbf_{a_k}^{n_{k_0}}$ for $k\geq k_0$, so $g\mycdot\Pbf=g\mycdot\Pbf_a^{n_{k_0}}=g\mycdot\Pbf_a^\infty$.
Hence we deduce
\[
\law_{\Pbf}\left(X^{\tau^{U_N}}\right) =\law_{\Pbf^\infty_a}\left(X^{\tau^{U_N}}\right),
\]
and letting $N\to\infty$ we deduce that $\Pbf=\Pbf_a^\infty$.
\end{proof}
\begin{lemma}[Continuity and Markov property]\label{lemContandM}
Let 
\[\begin{array}{ccc}
\N\cup\{\infty\}\times S^\Delta&\to&\Pcal(\Dloc(S))\\
(n,a)&\mapsto&\Pbf_a^n
\end{array}\]
be a weakly continuous mapping for the local Skorokhod topology such that $(\Pbf_a^n)_a$ is a Markov family for each $n\in\N$. Then $(\Pbf_a^\infty)_a$ is a Markov family.
\end{lemma}

Before giving the proof of the result recall the following property of the time change stated 
in the fifth part of Proposition 3.3 of \cite{GH017}: for any $(\Pbf_a)_a\in\Pcal(\Dloc(S))^S$ and $g\in\rmC(S,\R_+)$,
\begin{equation}\label{eqProp5TC}
(\Pbf_a)_a\text{ is $(\Fc_{t+})_t$-strong Markov }\Rightarrow(g\mycdot\Pbf_a)_a\text{ is $(\Fc_{t+})_t$-strong Markov }.
\end{equation}

\begin{proof}
Using Proposition \ref{propEqLocGlo}, there exists $g\in\rmC(S,\R_+^*)$ such that for all $(n,a)\in \N\cup\{\infty\}\times S^\Delta$, $\Pbf_a^n(\Dloc(S)\cap\D(S^\Delta))=1$ and such that $(n,a)\mapsto\Pbf_a^n$ is weakly continuous for the global Skorokhod topology from $\D(S^\Delta)$. For all $n\in\N$, by Lemma \ref{lmMpCimSM}, $(\Pbf_a^n)_a$ is $(\Fc_{t+})_t$-strong Markov, so, by \eqref{eqProp5TC}, $(g\mycdot\Pbf_a^n)_a$ is $(\Fc_{t+})_t$-strong Markov.

Take $a\in S$ and denote $\Tbb_a:=\big\{t\in\R_+\,\big| g\mycdot\Pbf^\infty_a(X_{t-}=X_t)=1\big\}$, so $\Tbb_a$ is dense in $\R_+$. Let $t\in\Tbb_a$ be and consider $F,G$ two bounded function from $\D(S^\Delta)$ to $\R$ continuous for the global Skorokhod topology, we want to prove that
\begin{align}\label{eq1proLemContandMark}
g\mycdot\Ebf_a^\infty\left[F\left((X_{t+s})_s\right)G\left((X_{t\wedge s})_s\right)\right]=g\mycdot\Ebf_a^\infty\Big[g\mycdot\Ebf_{X_t}^\infty[F]G\left((X_{t\wedge s})_s\right)\Big].
\end{align}
For any $n\in\N$, by the Markov property we have
\begin{align}\label{eq2proLemContandMark}
g\mycdot\Ebf_a^n\Big[F\left((X_{t+s})_s\right)G\left((X_{t\wedge s})_s\right)\Big]=g\mycdot\Ebf_a^n\Big[g\mycdot\Ebf_{X_t}^n[F]G\left((X_{t\wedge s})_s\right)\Big].
\end{align}
The mappings
\[\begin{array}{rcl}
\D(S^\Delta)&\to&\R\\
x&\mapsto&F\left((x_{t+s})_s\right)G\left((x_{t\wedge s})_s\right)
\end{array}\quad\text{ and }\quad\begin{array}{rcl}
\D(S^\Delta)&\to&\R\\
x&\mapsto&g\mycdot\Ebf_{x_t}^\infty[F]G\left((x_{t\wedge s})_s\right)
\end{array}\]
are continuous on the set $\{X_{t-}=X_t\}$ for the global topology. Hence, since $g\mycdot\Ebf_a^n$ converges to $g\mycdot\Ebf_a^\infty$ weakly for the global topology and $g\mycdot\Pbf^\infty_a(X_{t-}=X_t)=1$, we have
\begin{align}
\label{eq3proLemContandMark}&g\mycdot\Ebf_a^n\Big[F\left((X_{t+s})_s\right)G\left((X_{t\wedge s})_s\right)\Big]\cv{n\to\infty}g\mycdot\Ebf_a^\infty\Big[F\left((X_{t+s})_s\right)G\left((X_{t\wedge s})_s\right)\Big],\\
\label{eq4proLemContandMark}&g\mycdot\Ebf_a^n\Big[g\mycdot\Ebf_{X_t}^\infty[F]G\left((X_{t\wedge s})_s\right)\Big]\cv{n\to\infty}g\mycdot\Ebf_a^\infty\Big[g\mycdot\Ebf_{X_t}^\infty[F]G\left((X_{t\wedge s})_s\right)\Big].
\end{align}
Since $(n,b)\mapsto g\mycdot\Pbf_b^n$ is continuous for the global topology, using the compactness of $S^\Delta$ we have
\begin{align}\label{eq5proLemContandMark}
\sup_{a\in S^\Delta}\big|g\mycdot\Ebf_a^n F-g\mycdot\Ebf_a^\infty F\big|\cv{n\to\infty}0.
\end{align}
We deduce \eqref{eq1proLemContandMark} from 
\eqref{eq2proLemContandMark}-\eqref{eq5proLemContandMark}
and so
\[
g\mycdot\Ebf_a^\infty\big[F\left((X_{t+s})_s\right)\,\big|\,\Fc_t\big]=g\mycdot\Ebf_{X_t}^\infty[F],\quad g\mycdot\Pbf_a^\infty\text{-almost surely},
\]
so, by Lemma \ref{lmMpCimSM}, $(g\mycdot\Pbf_a^\infty)_a$ is $(\Fc_{t+})_t$-strong Markov. Applying \eqref{eqProp5TC} to $(g\mycdot\Pbf_a^\infty)_a$ and $1/g$, and using \eqref{eqProp2TC}, we deduce that  $(\Pbf_a^\infty)_a$ is $(\Fc_{t+})_t$-strong Markov.
\end{proof}
\begin{proof}[Proof of Theorem \ref{thmDefLFF}]\,

\emph{\ref{item1thmDefLFF}$\Rightarrow$\ref{item2thmDefLFF}}
Thanks to Proposition \ref{propEqLocGlo} there exists $g\in\rmC(S,\R_+^*)$ such that for all $a\in S^\Delta$, $\Pbf_a(\Dloc(S)\cap\D(S^\Delta))=1$ and such that the mapping $a\mapsto\Pbf_a$
is weakly continuous for the global Skorokhod topology from $\D(S^\Delta)$. Lemma \ref{lmMpCimSM} insure that $(\Pbf_a)_a$ is $(\Fc_{t+})_t$-strong Markov.  
By \eqref{eqProp5TC} we can deduce that $(g\mycdot\Pbf_a)_a$ is $(\Fc_{t+})_t$-strong Markov. Take $a\in S$ and $t\in\R_+^*$, we will prove that $g\mycdot\Pbf_a(X_{t-}=X_t)=1$. For any $f\in\rmC(S^\Delta)$, $s<t$ and $\eps>0$, by the Markov property
\[
g\mycdot\Ebf_a\Big[\frac{1}{\eps}\int_s^{s+\eps}f(X_u)\d u\,\Big|\,\Fc_s\Big]\overset{g\mycdot\Pbf_a\text{-a.s.}}{=}g\mycdot\Ebf_{X_s}\Big[\frac{1}{\eps}\int_0^{\eps}f(X_u)\d u\Big].
\]
Since $a\mapsto g\mycdot\Pbf_a$ is weakly continuous for the global topology and since $x\mapsto \frac{1}{\eps}\int_0^{\eps}f(X_u)\d u$ is continuous for the global topology,
\[
g\mycdot\Ebf_{X_s}\Big[\frac{1}{\eps}\int_0^{\eps}f(X_u)\d u\Big]\cv{\substack{s\to t\\s<t}}g\mycdot\Ebf_{X_{t-}}\Big[\frac{1}{\eps}\int_0^{\eps}f(X_u)\d u\Big].
\]
By a similar reasoning as in \eqref{eqPrLmMpCimSM} we have
\[
g\mycdot\Ebf_a\left|g\mycdot\Ebf_a\Big[\frac{1}{\eps}\int_t^{t+\eps}f(X_u)\d u\,\Big|\,\Fc_{t-}\Big]-g\mycdot\Ebf_a\Big[\frac{1}{\eps}\int_s^{s+\eps}f(X_u)\d u\,\Big|\,\Fc_s\Big]\right|\cv{\substack{s\to t\\s<t}}0
\]
so
\[
g\mycdot\Ebf_a\Big[\frac{1}{\eps}\int_t^{t+\eps}f(X_u)\d u\,\Big|\,\Fc_{t-}\Big]\overset{g\mycdot\Pbf_a\text{-a.s.}}{=}g\mycdot\Ebf_{X_{t-}}\Big[\frac{1}{\eps}\int_0^{\eps}f(X_u)\d u\Big].
\]
Hence letting $\eps\to 0$ we deduce $g\mycdot\Ebf_a\left[f(X_t)\mid\Fc_{t-}\right]\overset{g\mycdot\Pbf_a\text{-a.s.}}{=}f(X_{t-})$. Since $f$ is arbitrary, this is also true for $f^2$ so we deduce
\begin{align*}
g\mycdot\Ebf_a\left(f(X_t)-f(X_{t-})\right)^2&=g\mycdot\Ebf_a\left[g\mycdot\Ebf_a\left[f^2(X_t)\mid\Fc_{t-}\right]-f^2(X_{t-})\right]\\
&\quad-2g\mycdot\Ebf_a\left[f(X_{t-})\left(g\mycdot\Ebf_a\left[f(X_t)\mid\Fc_{t-}\right]-f(X_{t-})\right)\right]\\
&=0.
\end{align*}
Since $f$ is arbitrary, taking a dense sequence of $\rmC(S^\Delta)$, we get $g\mycdot\Pbf_a(X_{t-}=X_t)=1$. Finally, for any $t\in\R_+$ and $f\in\rmC(S^\Delta)$, since $x\mapsto f(x_t)$ is continuous for the global Skorokhod topology on $\{X_{t-}=X_t\}$, the function
\[\begin{array}{ccc}
S^\Delta&\to&\R\\
a&\mapsto&g\mycdot\Ebf_af(X_t)
\end{array}\]
is continuous, so $(g\mycdot\Pbf_a)_a$ is a Feller family.\\
\emph{\ref{item2thmDefLFF}$\Rightarrow$\ref{item3thmDefLFF}.}
Let $L$ be the $\rmC_0\times\rmC_0$-generator of $(g\mycdot\Pbf_a)_a$, then, by Proposition \ref{propUniMgPb}, $\Mc(L)=\{g\mycdot\Pbf_\mu\}_{\mu\in \Pcal(S^\Delta)}$ so by the first part of Proposition \ref{propLMPP} and by \eqref{eqProp2TC},
\[
\Mc\left(\frac{1}{g}L\right)=\{\Pbf_\mu\}_{\mu\in \Pcal(S^\Delta)}.
\]
\emph{\ref{item3thmDefLFF}$\Rightarrow$\ref{item1thmDefLFF}.} Thanks to \ref{it3PropLMPP} from Proposition \ref{propLMPP}, for the local Skorokhod topology,
\[\begin{array}{ccc}
\{\Pbf_a\}_{a\in S}&\to&S\\
\Pbf_a&\mapsto&a
\end{array}\]
is a continuous injective function defined on a compact set, so $a\mapsto\Pbf_a$ is also continuous. Let $\tau$ be a $(\Fc_{t+})_t$-stopping time and $a$ be in $S$. As in Remark \ref{rkExCL} we denote
\[
\Qbf_X\overset{\Pbf_a\text{-a.s.}}{:=}\law_{\Pbf_a}\left((X_{\tau+t})_{t\geq 0}\mid\Fc_{\tau+}\right).
\]
By using Proposition \ref{propCondLMP}, $\Qbf_X\in\Mc(L)$, $\Pbf_a$-almost surely, so $\Qbf_X=\Pbf_{X_\tau}$, $\Pbf_a$-almost surely, hence $(\Pbf_a)_a$ is $(\Fc_{t+})_t$-strong Markov. The quasi-continuity is a consequence of \ref{it4PropLMPP} from Proposition \ref{propLMPP}.\\
\emph{\ref{item2thmDefLFF}$\Rightarrow$\ref{item4thmDefLFF}.} Take an open subset $U\Subset S$ and define for all $a\in S$
\[
\widetilde{\Pbf}_a:= h\mycdot\Pbf_a\quad\text{ where }\quad h:=\frac{g\wedge\min_Ug}{\min_Ug}.
\]
By Proposition \ref{propTCFel}, $(\widetilde{\Pbf}_a)_a$ is Feller, and  moreover, since $X^{\tau^U}=(h\mycdot X)^{\tau^U}$,
\[
\forall a\in S,\quad\law_{\Pbf_a}\left(X^{\tau^U}\right) = \law_{\widetilde{\Pbf}_a}\left(X^{\tau^U}\right).
\]
\emph{\ref{item4thmDefLFF}$\Rightarrow$\ref{item1thmDefLFF}.} Let $U_n\Subset S$ be an increasing sequence of open subsets such that $S=\bigcup_nU_n$. For each $n\in\N$ there exists a Feller family $(\Pbf^n_a)_a$ such that
\[
\forall a\in S,\quad\law_{\Pbf_a}\left(X^{\tau^{U_n}}\right) = \law_{\Pbf^n_a}\left(X^{\tau^{U_n}}\right).
\]
Denote $\Pbf^\infty_a:=\Pbf_a$, then thanks to Lemma \ref{lemLocCont} the mapping
\[\begin{array}{ccc}
\N\cup\{\infty\}\times S^\Delta&\to&\Pcal(\Dloc(S))\\
(n,a)&\mapsto&\Pbf_a^n
\end{array}\]
is continuous and thanks to Lemma \ref{lemContandM} $(\Pbf^\infty_a)_a$ is a Markov family.
\end{proof}

\subsection{Proof of Lemma \ref{lmMPtoSMP}}\label{secProofLmMPtoSMP}

Before proving the Lemma \ref{lmMPtoSMP} let us note that
thanks to Proposition \ref{propEqLocGlo} and \eqref{eqProp5TC}, if $(\Pbf_a)_a\in\Pcal(\Dloc(S))^S$ is locally Feller then for any open subset $U\subset S$ there exists $h_0\in\rmC(S,\R_+)$ such that $U=\{h_0\not = 0\}$ and $(h_0\mycdot\Pbf_a)_a$ is locally Feller. This fact does not ensure that the martingale local problem associate to $h_0L$ is well-posed as is stated in lemma. During the proof we will use two preliminary results.
\begin{lemma}\label{lmExTCWP}
Let $L$ be a subset of $\rmC_0(S)\times\rmC(S)$ such that $\rmD(L)$ is dense in $\rmC_0(S)$ and $U$ be an open subset of $S$, then there exist a subset $L_0$ of $L$ and a function $h_0$ of $\rmC(S,\R_+)$ with $\{h_0\not =0\}=U$ such that $\overline{L}=\overline{L_0}$, such that $h_0L_0\subset\rmC_0(S)\times\rmC_0(S)$ and such that: for any $h\in\rmC(S,\R_+)$ with $\{h\neq 0\}=U$ and $\sup_{a\in U} (h/h_0)(a)<\infty$ and any $\Pbf\in\Mcons\left((hL_0)^\Delta\right)$, $\Pbf(X=X^{\tau^U})=1$.
\end{lemma}
\begin{proof}
Take $L_0$ a countable dense subset of $L$ and let $d$ be a metric on $S^\Delta$. For any $n\in\N^*$ there exist $M_n\in\N$ and $(a_{n,m})_{1\leq m\leq M_n}\in (S^\Delta\backslash U)^{M_n}$ such that
\[
S^\Delta\backslash U\subset\bigcup_{m=1}^{M_n}B(a_{n,m},n^{-1}).
\]
For each $1\leq m\leq M_n$ there exists $(f_{n,m},g_{n,m})\in L_0$ such that
\[
f_{n,m}(a)\in\left\{\begin{array}{ll}
[1-n^{-1},1+n^{-1}] & \text{if } d(a,a_{n,m})\geq 2n^{-1},\\
{[-n^{-1},1+n^{-1}]} & \text{if } n^{-1}\leq d(a,a_{n,m})\geq 2n^{-1},\\
{[-n^{-1},n^{-1}]} & \text{if } n^{-1}\leq d(a,a_{n,m}).
\end{array}\right.
\]
Take $h_0\in\rmC_0(S,\R_+)$ with $\{h_0\not =0\}=U$, such that $h_0g\in\rmC_0(S)$ for any $(f,g)\in L_0$ and such that for any $n\in\N^*$ and $1\leq m\leq M_n$
\[
\|h_0\|_{B(a_{n,m},4n^{-1})}\|g_{n,m}\|\leq \frac{1}{n}.
\]
Hence $\overline{L}=\overline{L_0}$ and $hL_0\subset\rmC_0(S)\times\rmC_0(S)$.
Let $h\in\rmC(S,\R_+)$ be such that $\{h\neq 0\}=U$ and $C:=\sup_{a\in U} (h/h_0)(a)<\infty$.
Let $\Pbf\in\Mcons\left((hL)^\Delta\right)$ be such that there exists $a\in S^\Delta\backslash U$ with $\Pbf(X_0=a)=1$. We will prove that
\begin{align}\label{eqLmExTCWP}
\Pbf(\forall s\geq 0,~X_s=a)=1.
\end{align}
Take $t\in\R_+$ and $n\in\N$. There exists $m\leq M_n$ such that $d(a,a_{n,m})<\frac{1}{n}$. If we denote
\[
\tau:=\tau^{B(a,3n^{-1})},
\]
then
\begin{align*}
\Ebf[f_{n,m}(X_{t\wedge\tau})]
&=f_{n,m}(a)+\Ebf\left[\int_0^{t\wedge\tau}h(X_s)g_{n,m}(X_s)\d s\right]\\
&\leq f_{n,m}(a)+t\|h\|_{B(a_{n,m},4n^{-1})}\|g_{n,m}\|
\leq \frac{1+tC}{n}
\end{align*}
Since by \ref{it4PropLMPP} from Proposition \ref{propLMPP} we have $\Pbf\big(\tau<\infty \Rightarrow d(X_\tau,a)\geq\frac{3}{n}\big)=1$,
\begin{align*}
\Ebf[f_{n,m}(X_{t\wedge\tau})]
&=\Ebf[f_{n,m}(X_{\tau})\mathds{1}_{\{\tau\leq t\}}]+\Ebf[f_{n,m}(X_{t})\mathds{1}_{\{t<\tau\}}]\\
&\geq (1-\frac{1}{n})\Pbf(\tau\leq t)-\frac{1}{n}\Pbf(t<\tau)
=\Pbf(\tau\leq t)-\frac{1}{n},
\end{align*}
so
\[
\Pbf(\tau\leq t)\leq \frac{2+tC}{n}.
\]
Hence we obtain
\[
\Pbf\big(\forall s\in[0,t],~d(X_s,a)\leq \frac{3}{n}\big)\geq\Pbf(t<\tau)\geq 1-\frac{2+tC}{n}.
\]
By taking the limit with respect to $n$ and $t$ we obtain \eqref{eqLmExTCWP}.

To complete the proof let us consider an arbitrary $\Pbf\in\Mcons\left((hL_0)^\Delta\right)$. As in Remark \ref{rkExCL} we denote 
\[
\Qbf_X\overset{\Pbf\text{-a.s.}}{:=}\law_\Pbf\left((X_{\tau^U+t})_{t\geq 0}\mid\Fc_{\tau^U}\right). 
\]
Thanks to Proposition \ref{propCondLMP} $\Pbf$-almost surely $\Qbf_X\in\Mc\left((hL)^\Delta\right)$, and thanks to \ref{it4PropLMPP} from Proposition \ref{propLMPP} $\Pbf$-almost surely $\Qbf_X(X_0=a)=1$ with $a=X_\tau\in S^\Delta\backslash U$ on $\{\tau^U<\infty\}$. By using the previous situation and by applying \eqref{eqLmExTCWP} we get that $\Pbf$-almost surely $\Qbf_X(\forall s\geq 0,~X_s=a)=1$, with $a=X_\tau\in S^\Delta\backslash U$ on $\{\tau^U<\infty\}$. Hence $\Pbf(X=X^{\tau^U})=1$.
\end{proof}
\begin{lemma}\label{lmMPtoSMP1}
Let $L$ be a subset of $\rmC_0(S)\times\rmC_0(S)$ such that the martingale problem associated to $L$ is well-posed. Then the martingale problem associated to $L^\Delta$ is well-posed if and only if $\Pbf(X=X^{\tau^S})=1$ for all $\Pbf\in\Mcons(L^\Delta)$ (in other words $\Pbf\in\Pcal(\Dloc(S))$) .
\end{lemma}
\begin{proof}
Assume that the martingale problem associated to $L^\Delta$ is well-posed and take $\Pbf\in\Mcons(L^\Delta)$. Then $\law_\Pbf(X^{\tau^S})\in\Mcons(L^\Delta)$, so by uniqueness of the solution $\Pbf=\law_\Pbf(X^{\tau^S})$ and so $\Pbf(X=X^{\tau^S}) = 1$. For the converse, 
let $\Pbf^1,\Pbf^2\in\Mcons(L^\Delta)$ be such that $\law_{\Pbf^1}(X_0)=\law_{\Pbf^2}(X_0)$. Then $\Pbf^1,\Pbf^2\in\Pcal(\Dloc(S))$ so $\Pbf^1,\Pbf^2\in\Mc(L)$, hence $\Pbf^1=\Pbf^2$.
\end{proof}
\begin{proof}[Proof of Lemma \ref{lmMPtoSMP}]
Let $L_0$ and $h_0$ be as in Lemma \ref{lmExTCWP} and take $h\in\rmC(S,\R_+)$ with $\{h\neq 0\}=U$ and $\sup_{a\in U} (h/h_0)(a)<\infty$. The existence of a solution for the martingale problem associated to $(hL_0)^\Delta$ is given 
by the existence of a solution for the martingale problem associated to $L$. Let $\Pbf^1,\Pbf^2\in\Mcons((hL_0)^\Delta)$ be such that $\law_{\Pbf^1}(X_0)=\law_{\Pbf^2}(X_0)$. Thanks to Lemma \ref{lmExTCWP} and Lemma \ref{lmMPtoSMP1}, for an open subset $V\Subset U$, there exist $k\in\rmC(S,\R_+^*)$ and a dense subset $L_1$ of $L_0$ such that $k(a)=h(a)$ for any $a\in V$, $kL_1\subset\rmC_0(S)\times\rmC_0(S)$ and the martingale problem associated to $(kL_1)^\Delta$ is well-posed.
Hence we may apply Theorem 6.1 p. 216 from \cite{EK86} and deduce that $\law_{\Pbf^1}(X^{\tau^V})=\law_{\Pbf^2}(X^{\tau^V})$. Letting $V$ 
growing toward $U$ we deduce that $\law_{\Pbf^1}(X^{\tau^U})=\law_{\Pbf^2}(X^{\tau^U})$ and so, since $\Pbf^i(X=X^{\tau^U})=1$ for $i\in\{1,2\}$, we conclude that $\Pbf^1=\Pbf^2$.
\end{proof}

\bibliographystyle{alpha}
\bibliography{Loc_Fel.bib}

\begin{thebibliography}{BSW13}

\bibitem[BS11]{BS11}
Bj\"orn B\"ottcher and Alexander Schnurr.
\newblock The {E}uler scheme for {F}eller processes.
\newblock {\em Stoch. Anal. Appl.}, 29(6):1045--1056, 2011.

\bibitem[BSW13]{BSW13}
Bj\"orn B\"ottcher, Ren\'e Schilling, and Jian Wang.
\newblock {\em L\'evy matters. {III}}, volume 2099 of {\em Lecture Notes in
  Mathematics}.
\newblock Springer, Cham, 2013.
\newblock L\'evy-type processes: construction, approximation and sample path
  properties, With a short biography of Paul L\'evy by Jean Jacod, L\'evy
  Matters.

\bibitem[EK86]{EK86}
Stewart~N. Ethier and Thomas~G. Kurtz.
\newblock {\em Markov processes}.
\newblock Wiley Series in Probability and Mathematical Statistics. John Wiley
  \& Sons, Inc., New York, 1986.
\newblock Characterization and convergence.

\bibitem[GH17a]{GH217}
Mihai Gradinaru and Tristan Haugomat.
\newblock L\'evy-type processes: convergence and discrete schemes.
\newblock {\em arXiv: 1707.02889}, July 2017.

\bibitem[GH17b]{GH017}
Mihai Gradinaru and Tristan Haugomat.
\newblock Local {S}korokhod topology on the space of cadlag processes.
\newblock {\em arXiv: 1706.03669}, June 2017.

\bibitem[Kal02]{Ka02}
Olav Kallenberg.
\newblock {\em Foundations of modern probability}.
\newblock Probability and its Applications. Springer-Verlag, New York, second
  edition, 2002.

\bibitem[Kur11]{Ku11}
Thomas~G. Kurtz.
\newblock Equivalence of stochastic equations and martingale problems.
\newblock In {\em Stochastic analysis 2010}, pages 113--130. Springer-Verlag,
  Berlin, Heidelberg, 2011.

\bibitem[Lum73]{Lu73}
Gunter Lumer.
\newblock Perturbation de g\'en\'erateurs infinit\'esimaux, du type
  ``changement de temps''.
\newblock {\em Ann. Inst. Fourier (Grenoble)}, 23(4):271--279, 1973.

\bibitem[Str75]{St75}
Daniel~W. Stroock.
\newblock Diffusion processes associated with {L}\'evy generators.
\newblock {\em Z. Wahrscheinlichkeitstheorie und Verw. Gebiete},
  32(3):209--244, 1975.

\bibitem[SV06]{SV06}
Daniel~W. Stroock and S.~R.~Srinivasa Varadhan.
\newblock {\em Multidimensional diffusion processes}.
\newblock Classics in Mathematics. Springer-Verlag, Berlin, 2006.
\newblock Reprint of the 1997 edition.

\bibitem[vC92]{vC92}
Jan~A. van Casteren.
\newblock On martingales and {F}eller semigroups.
\newblock {\em Results Math.}, 21(3-4):274--288, 1992.

\end{thebibliography}
\end{document}